\documentclass{amsart}

\usepackage{amsthm, amssymb,amsmath}
\usepackage{mathscinet}
\usepackage[utf8]{inputenc}
\usepackage{hyperref}
\usepackage{graphicx}
\usepackage{subcaption}
\usepackage{dsfont}
\usepackage{rotating}
\newtheorem{theorem}{Theorem}[section]
\newtheorem{lemma}[theorem]{Lemma}
\newtheorem{corollary}[theorem]{Corollary}
\newtheorem{proposition}[theorem]{Proposition}
\newtheorem{claim}[theorem]{Claim}
\newtheorem{remark}[theorem]{Remark}
\newtheorem{definition}[theorem]{Definition}

\numberwithin{equation}{section}

\DeclareMathOperator{\SL}{SL}

\DeclareMathOperator{\GL}{GL}

\DeclareMathOperator{\Dim}{Dim}
\DeclareMathOperator{\Hess}{Hess}

\newcommand{\R}{\mathbb{R}}
\newcommand{\C}{\mathbb{C}}
\newcommand{\Z}{\mathbb{Z}}
\newcommand{\N}{\mathbb{N}}

\newcommand{\E}{\mathbb{E}}
\newcommand{\SSSS}{\mathbb{S}}
\newcommand{\PP}{\mathbb{P}}
\newcommand{\F}{\mathcal{F}}

\newcommand{\A}{\mathcal{A}}
\newcommand{\B}{\mathcal{B}}
\newcommand{\CC}{\mathcal{C}}
\newcommand{\DD}{\mathcal {D}}
\newcommand{\RR}{\mathcal{R}}

\newcommand{\W}{\mathcal{W}}
\newcommand{\EE}{\mathcal{E}}

\newcommand{\ZZ}{\mathcal{Z}}

\newcommand{\MMM}{\mathfrak{m}}

\newcommand{\om}{\omega}
\newcommand{\s}{\sigma}
\newcommand{\g}{\gamma}
\newcommand{\G}{\Gamma}
\newcommand{\de}{\delta}
\newcommand{\De}{\Delta}

\newcommand{\la}{\lambda}
\newcommand{\La}{\Lambda}
\newcommand{\vf}{\varphi}

\newcommand{\e}{\varepsilon}

\newcommand{\Si}{\Sigma}
\newcommand{\Om}{\Omega}

\newcommand{\wt}{\widetilde}
\newcommand{\wh}{\widehat}

\newcommand{\x}{\times}

\newcommand{\ov}{\overline}

\begin{document}

\title{Invariance Principle and Dynamics}
\author{Fran\c cois Ledrappier}
\address{Fran\c cois Ledrappier, Universit\'e de Paris et Sorbonne Universit\'e, CNRS, LPSM, Bo\^{i}te Courrier 158, 4, Place Jussieu, 75252 PARIS cedex
05, France,} \email{fledrapp@nd.edu}

\begin{abstract} This is a survey of some invariances that arise in dynamical systems theory in the presence of zero Lyapunov exponents. \end{abstract}

\maketitle 

\section{Introduction} 

An active aspect of the theory of dynamical systems is the study of asymptotic properties of successive iterations of a transformation. This is performed through the description of global 
(often statistical) properties and objects.

For instance, the Lyapunov exponents were introduced by A.M. Lyapunov to study the stability of dynamical properties under perturbations. In the late sixties, in foundational papers about smooth diffeomorphisms and flows on compact manifolds, D.V. Anosov  and S. Smale showed  that uniform absence of vanishing exponents also induces some form of stability (see \cite{M88} for the converse).  On the other hand, at about the same time, D.V. Oseledets  showed that Lyapunov exponents always exist in a probabilistic sense, almost everywhere for any invariant probability measure. The systems with no vanishing Lyapunov exponents in this sense  have a rich theory as well, see \cite{Y13},  \cite{W17}. The question so arises of understanding vanishing Lyapunov exponents of invariant probability measures. It turns out that, often, vanishing Lyapunov exponents implies a non-trivial invariance of... something. These results are nicknamed  as  the action of an Invariance Principle\footnote{Our IP is not related  with Donsker IP in Large Deviations theory. Ours came later, but it is more of  a principle that has something to do with some  invariance.}  and the goal of this paper is to survey the ergodic theory behind some examples. In this text, we write IP for Invariance Principle.

The common feature of these examples is the introduction of some entropy, typically an average of Kullback-Liebler informations, that quantifies  how some measures are distorted by the dynamics. Then the vanishing Lyapunov exponent forces this entropy to vanish as well (see the basic inequalities propositions \ref{basicinequality1}, \ref{basicinequality2}, \ref{basicinequality3}, \ref{basicinequality4}). The vanishing of this entropy forces the desired invariance by an application of the equality case in Jensen inequality. 

In the  section \ref{examples}, we recall more or less classical results that led to the formulation of the IP. The story starts with H. Furstenberg criterion for a  random walks on $\SL(d,\R)$ to have non-zero exponents (\cite{F63}). For dependent product of matrices, the case of  Markovian dependence was considered in \cite{V79}, \cite{G79} and \cite{R80}  and much weaker dependence was considered by S. Kotani and B. Simon (\cite{S83}) for the study  of Schr\"odinger operators. The form of dependence in theorem \ref{IP} for linear cocycles was formulated in \cite{L86}. The non-linear version pertains to independent products of diffeomorphisms of compact manifolds and was proven by H. Crauel (\cite{C90}).  Further recent examples are a result of A. Tahzibi and J. Yang (\cite{TY19}) about dynamically coherent partially hyperbolic systems with compact center leaves and hyperbolic quotient (see section \ref{PHcompact}) and a result of S. Crovisier and M. Poletti about systems with quasi-isometric center leaves (see section \ref{QIcenter}).

The IP is the realization that all these results are manifestations of the same phenomenon, that we explain in sections 3 and \ref{IPbis}. The complete nonlinear version, the catchy name and many applications are due to A. Avila and M. Viana and we follow  \cite {AV10}. We give a few examples in sections \ref{examples2} and \ref{examples3}.   Section 7 is an incomplete laundry list of other related results.  In general, the proofs are only sketched, and we invite the interested reader to go back to the original papers. We hope that such an overview can help them.
There is a considerable overlap with another survey  on the same topic (and with almost the same title) by Karina Marin and Mauricio Poletti (\cite{MP25}), which presents a synthetic view of the current research.  \cite{MP25} is more precise and complete on the examples of applications. We present the ergodic theoretic background of the IP in the appendix. We also refer to Marcelo Viana's book (\cite{V14}) about Lyapunov exponents and Romain Dujardin's \cite{D26} about ergodic theory for more background and more details about some of the results mentioned here.

This survey is the  extended version of  a mini-course I gave in the September 2025  Meeting of the AnR DynAtrois in Dijon. I thank the organizers for the invitation and the audience for 
their attention and remarks. I also thank Y. Coud\`ene, S.~Crovisier, R. Dujardin, M. Poletti,  A. Wilkinson and the anonymous referees for many relevant  suggestions. 

\

\section {Examples of applications of an invariance principle}\label{examples}

\subsection {Product of independent $2\x2$ matrices. Exponents. Furstenberg criterion}

\

Let \( \mu\) be a probability  measure on  \( \SL (2, \R) \), with \( \int \log \|A\| \, d\mu (A) < +\infty .\)

We form the {\it {independent product }} of matrices with distribution \(\mu .\) Formally, set 
 \( (\ov X, \ov \A, \ov m) \) for  the product space \(\ov X :=  ( \SL(2,\R))^{\otimes \N },\) with the $\s$-algebra \( \ov \A\) generated by the coordinates \( \{x_n \}_{n\ge 0},\) and the product measure \( \ov m := \mu ^{\otimes \N} .\) For \( \{x_n \}_{n\ge 0} \in \ov X, \, n\ge 1 \) set \[ Z_n := x_{n-1} \ldots x_0.\] 
\begin{theorem}[\cite{FK60}]\label{FK} There is a number \( \g \ge 0 \) such that, for \( \ov m \)-a.e. \(x \in \ov X,\) 
\[ \g  \,=\, \lim\limits _{n \to \infty } \frac{1}{n} \log \| Z_n \|  \quad =\, \lim\limits _{n \to \infty } \frac{1}{n} \log \| Z_n^{-1} \| .\] \end{theorem}
Theorem \ref{FK} anticipates the subadditive ergodic theorem. We have \( \g \ge 0 \) because \( Z_n Z_n ^{-1} = \mathbb Id .\)
The original IP is 

\begin{theorem}[Furstenberg \cite{F63}] \label{F}  If \( \g = 0\) , then there exists a probability \( \nu \) on \( \PP ^1(\R) \) which is {\it {invariant}} under \( \mu \)-a.e. matrix $A$, where \( A \in \SL (2,\R)\) acts naturally on \( \PP^1(\R):\)  \(A_\ast \nu = \nu \) for \(\mu \)-a.e. $A.$ \end{theorem}

By discussing how different matrices may have the same invariant probability  measure for their actions on \(\PP^1(\R),\) we can see from theorem \ref{F} that if \( \g =0,\) then either 
 there is a finite set of directions invariant under \( \mu \)-a.e. matrix $A$, or  the closed subgroup generated by the support of \( \mu \) is compact. We proof theorem \ref{F} in this section.

\begin{definition}
The probability  measure $\nu $ on  \( \PP ^1(\R) \) is called {\bf{stationary }} if \[ \nu \; =\; \int A_\ast \nu \, d\mu (A).\] \end{definition}
Observe that there  always exist stationary probability measures \( \nu \) on \( \PP ^1(\R) . \) The proof of theorem \ref{F} consists in showing that if \( \g =0,\) then stationary measures have to be invariant.
The tool is Furstenberg entropy.

\begin{definition}\label{Fentropy}
Let \( \mu\) be a probability measure on  \( \SL (2, \R) \), \( \nu \) a stationary probability  measure on \(\PP ^1 (\R).\) The {\bf {Furstenberg entropy}} \( \kappa (\mu, \nu)\) is defined by  \[ \kappa (\mu, \nu ) \;:= \; \int \log \frac {dA_\ast \nu }{d\nu } (A\xi ) \, d\nu (\xi) \, d\mu (A).\] \end{definition}
We set \(  \kappa (\mu, \nu )  = + \infty \) if, for a set of matrices $A$ of positive \(\mu\)-measure, the measure \( A_\ast \nu \) is not absolutely continuous with respect to \(\nu.\) We also have, a priori, \(  \kappa (\mu, \nu )  = + \infty \) if the function \( \log \frac {dA_\ast \nu }{d\nu } (A\xi ) \) is not integrable (by Jensen inequality, the integral, finite or infinite, makes sense). Again by Jensen
inequality, we have:

\begin{proposition}\label{invariance1} We have  \( \kappa (\mu, \nu) \ge 0 \)  and \( \kappa (\mu, \nu ) =0 \)  if, and only if, \(\nu \)  is invariant  under \( \mu \)-a.e. matrix $A$. \end{proposition} 
\begin{proof} Indeed, there is equality in Jensen inequality if, and only if,  \(\displaystyle \frac {dA_\ast \nu }{d\nu } (\xi ) = 1 \; (d(A_\ast \nu )(\xi)  d\mu (A) )\)-a.e.. \end{proof}
The main step in the proof is the following 
\begin{proposition}[{\bf{Basic inequality}}]\label{basicinequality1} Let \( \mu\) be a probability  measure on  \( \SL (2, \R) \), with \( \int \log \|A\| \, d\mu (A) < +\infty , \) \( \nu \) a stationary probability  measure on \(\PP ^1 (\R).\)  Then,  \[ \kappa (\mu, \nu )  \; \le \; 2 \g .\]\end{proposition} 

\begin{proof}  
Consider \( (\ov \EE,  \ov M)  :=  (\ov X \times \PP^1(\R), \ov m \times \nu ).\)
We have \[  \kappa (\mu, \nu ) \;:= \; \int \log \frac {d(x_0)_\ast \nu }{d\nu } (x_0\xi ) \, d\ov M (x,\xi).\]
 Consider the transformation   \( \ov F: \ov \EE \to \ov \EE ,\) defined by  \(  \ov F (x, \xi ) \;:= \; (\ov f(x), x_0 \xi)\), where $\ov f$ is the shift on the product space \( \ov X.\) We have  \( \ov F ^j (x,\xi ) = (\ov f^j x, Z_j (x) \xi).\)

Observe that the measure \( \ov M\) is invariant under the transformation \( \ov F:\)  for all \( \vf \in C(\ov \EE),\) we can write, using the stationarity relation,
\[ \int \vf \circ \ov F\, d\ov M \;= \;\int \vf (x_1, \ldots, x_0 \xi ) d\nu (\xi ) d\mu (x_0) d\mu (x_1) \dots \; =\;  \int \vf \, d\ov M.\]

By a general result of Breiman about Markov chains (cf. e.g. \cite[ Theorem I.2.1)]{K86} the ergodic decomposition of \(\ov M\) is given by the extremal decomposition of \(\nu\) into stationary measures. So we first  assume that the measure \(\ov M\) is ergodic for \(\ov F \) and, equivalently, that the measure \(
\nu\) is extremal among stationary measures. 

  Write, for $ x\in \ov X, \xi \in \SSSS^1,$ 
\[ g( x, \xi) \;:= \; \frac {d(x_0^{-1})_\ast \nu} {d\nu}(\xi)\]
and, for  $N$ large, \( g^N ( x,\xi) := \sup \{g(x, \xi), e^{-N}\} .\) The function $-\log g^N $ is $\ov M$-integrable, \( -\log g \ge - \log g^N \)  and, by monotone limit,\footnote{Here, and in all similar formulas in the paper, both terms might be  a priori infinite. The proof  of  proposition \ref{basicinequality1} will in particular show that, under our hypotheses, $\kappa (\mu, \nu ) < + \infty .$ (see corollary \ref{abs.cont.})} \begin{equation}\label{infinite0}   \kappa (\mu, \nu ) \; = \; \lim\limits _{N \to \infty } - \int \log g^N (x,\xi) \, d\ov M (x,\xi).\end{equation}

Fix $\e >0$ and choose  $\de >0, \, \de = \de (N, \e),$ such that, setting 
\begin{eqnarray*}  g_\de ( x,\xi) &:= & \frac {\int _{B(\xi, \de)} g( x,\eta) \, d\nu}{\nu (B(\xi,\de))} =  \frac {\nu ( x_0(B(\xi, \de)))}{\nu (B(\xi,\de))},\\
 g_\de^N ( x,\xi) &:= & \frac {\int _{B(\xi, \de)} g^N( x,\eta) \, d\nu  (\eta)}{\nu(B(\xi,\de))} , \quad 
 g_\de^\ast := \sup _{\de ' \le \de } g_{\de '}^N,\end{eqnarray*}
we have $\displaystyle \int \log \frac {g_\de^\ast}{g^N} \, d\ov M  \le \e$ (see proposition \ref{densities}). Observe that, for $q\le \de ,$ \[-\log g^N \le -\log g_q +\log g_q -  \log g_q^N +\log g^\ast_\de - \log g^N  \; \le \; -\log g_q +\log g^\ast_\de - \log g^N .  \]
Set, for $j\ge 0,$   $\displaystyle q_j := \frac{ \de }{\| x_0\|^2 \ldots  \| x_{j-1}\|^2} .$ In particular, $q_j \le \de $ for all $j \ge 0.$

By the ergodic theorem, we may  write, at $\ov M$-a.e.$(\ov x,\xi)$,   \begin{eqnarray*} -\int \log g^N \, d \ov M &=& \lim_{n\to \infty } \frac{1}{n}\sum _{j=0}^{n-1} -\log g^N(\ov F^j (x,\xi)) \\
 &\le & \liminf_{n\to \infty } \frac{1}{n}\sum _{j=0}^{n-1} -\log g_{q_j}(\ov F^j (x,\xi)) \; +\; \int \log \frac {g_\de^\ast}{g^N} \, d\ov M  \\
  &\le & \liminf_{n \to \infty } \frac{1}{n}\sum _{j=0}^{n-1} -\log \frac{\nu (x_j (B(Z_j(x)\xi, q_j)))}{ \nu (B(Z_j(x)\xi, q_j))}\; +\; \e . \end{eqnarray*}
Observe that \(  x_j (B(Z_j(x)\xi, q_j)) \supset B(Z_{j+1}(x)\xi, \frac{q_{j}}{\|x_j\|^2}) =  B(Z_{j+1}(x)\xi, q_{j+1}).\) The sum telescopes and we get, at $\ov M$-a.e.$( x,\xi)$,
\[-\int \log g^N \, d \ov M  \; \le \;  \liminf_{ n \to \infty }- \frac{1}{n}\log \frac{\nu (B(Z_n(x)\xi, q_n))}{ \nu (B(\xi, \de))} + \e, \]

By proposition \ref{densities}, there are  constants $C$ and $\g $  such that
 the  set $D $ of points in $\SSSS^1$ such that $\nu ( B(\xi, \g')) \ge C  \g' $  for all $0< \g' \le \g $ has $\nu$-measure at least $1-\e$ Therefore, for $\ov M$-a.e.$( x,\xi),$ there are infinitely many indices $n$ such that $q_n<  \g  $ and \(  \ov F^{(n )} ( x,\xi)  =( \ov f^n x, Z_n(x)\xi ) \in \ov X \x D .\) 
 It follows that, for all $\e >0,$ for $\ov M$-a.e.$( x,\xi),$
 \[ \liminf_{n \to \infty }- \frac{1}{n}\log \frac{\nu ((B(Z_n(x)\xi)), q_n))}{ \nu (B(\xi, \de))} \;\le\; \lim\limits _{n\to \infty} - \frac{1}{n} \log q_n   = 2\int \log \| A\| \, d\mu (A) .\]
 Letting $\e \to 0$ and then  using  (\ref{infinite0}), we obtain \begin{equation}\label{presquebasic} \kappa (\mu , \nu ) \le 2 \int \log \| A\| \, d\mu (A).\end{equation}
 
If the measure $\nu $ is not extremal, it admits a decomposition $\nu = \int \nu _t \, dn(t) $ in extremal measures and then $\kappa (\mu , \nu ) \le \int \kappa (\mu, \nu _t) \, dn (t).$ The inequality (\ref{presquebasic}) still holds for a general stationary $\nu $. 
Let $\mu ^{\ast L} $ be the distribution of $Z_L$. The measure $\nu$ is stationary for $\mu ^{\ast L} $ as well and $\kappa (\mu ^{\ast L} , \nu ) = L \kappa (\mu, \nu ) .$
It follows that  \[ \kappa (\mu , \nu ) \le 2 \lim\limits _{L \to \infty} \frac{1}{L} \int \log \| Z_L\| \, d\mu ^{\ast L}  = 2\g .\]  \end{proof}

 \begin{remark} The original proof of proposition \ref{basicinequality1}  consists in approximating $\mu $ by measures $\mu _n$ with an absolutely continuous stationary measure $\nu _n$. There is equality when the stationary measure  is absolutely continuous and the inequality goes to the limit (see \cite{L84}). The proof we present uses the ideas of the proof of proposition \ref{basicinequality4} in a simpler setting. If one considers the action of $\PP\SL(2,\R) $ on the hyperbolic space by isometries, then the exponent $\g $ is also given by $\g = \lim\limits _{n\to \infty } \frac{1}{2n} d(o, Z_no) $ and then the basic inequality goes back to Guivarc'h (\cite{guivarch}) if $\mu$ is supported on a discrete group. Another, classical, instance of the basic inequality is the Margulis-Ruelle inequality in $C^1$ dynamics (see \cite{R79}, or, more precisely for us, \cite{K86} and \cite{BB95} in the random dynamics case). \end{remark}

\begin{proof}({\it {Proof of theorem \ref{F}}}) Theorem \ref{F} clearly follows from propositions \ref{invariance1} and \ref{basicinequality1}. \end{proof}

Proposition \ref{basicinequality1} and Theorem \ref{F} readily extend to \(\SL (d,\R)\), for all $d\ge 2: $ we have \( \kappa (\mu ,\nu ) \le (d-1) (\g ^+ -\g ^- ), \) where, for \(\ov m\)-a.e. \( x \in (\SL(d,\R)^\N, \) 
 \[ \g ^+ \,:=\, \lim\limits _{n \to \infty } \frac{1}{n} \log \| Z_n (x)\|,  \quad  \g ^- :=\, -\lim\limits _{n \to \infty } \frac{1}{n} \log \| Z_n^{-1}(x) \| .\]
 \begin{corollary}\label{abs.cont.} Let $\mu $ be a probability measure on $\SL(2,\R)$, such that \( \int \log \|A\| \, d\mu (A) < +\infty ,\) $\nu $ a stationary probability measure on $\SSSS ^1.$ Then, for $\mu $-a.e. $A$, the measure $(A^{-1})_\ast \nu $ is absolutely continuous with respect to $\nu ,$ and $\int \log \frac{d (A^{-1})_\ast \nu}{d\nu} (\xi ) \, d\nu (\xi) > -\infty .$ \end{corollary}
 \subsection{ Application to discrete 1-d Schr\"odinger equation.}
 
 One of the first applications of Furstenberg theorem \ref{F} is to one-dimensional Anderson model in theoretical physics.
 
  Let \( \{V_n\}_{n\in \Z} \) be a sequence of bounded independent real variables with distribution \( \tau\), \( (X,m) := (\R^\Z, \tau ^{\otimes \Z}).\) The one-dimensional Schr\"odinger operator at \( x= \{V_n\}_{n\in \Z} \in X\) is the self-adjoint operator \(H_x\) of \( \ell ^2 (\Z) \) given for \( u = \{u_n \}_{n \in \Z} \in \ell ^2 (\Z), \)  for all \( n \in \Z,\) by 
\[\left( H_x (u)\right)_n \;:= \; u_{n+1} + u_{n-1} + V_n u_n.\]
The operator \( H_x \) admits a spectral measure $\s_x$ on $\R$, given by its moments  \(\int E^n \, d\s_x(E) = \frac{1}{2}\left( (H_x(\de_0))_n +(H_x(\de _1))_n \right) .\) By ergodicity of the shift transformation $f$ on $X$, the type\footnote{The {\bf {type}} of a positive measure on $\R$ is whether the measure is discrete, singular continuous or absolutely continuous with respect to the Lebesgue measure or a convex combination of these.} and the support of \( \s_x \) are $m$-a.e. constant.

In the one-dimensional  Anderson  model, \( \exp (it H_x)u \) describes the quantum evolution of the quantum object \(u\). The usual understanding  is that eigenfunctions (i.e. the discrete part of \(\s_x\)) correspond to {\bf {localization}}, the absolutely continuous part of \(\s_x\)  to {\bf {diffusion}}.
  
  \begin{theorem}  [\cite{I73}]
 For \(m\)-a.e. \(x\in X,\) the spectral measure \( \s_x \) is singular with respect to the Lebesgue measure on $\R$. \end{theorem}
 \begin{proof} (See \cite{P80}) Consider the measure \( \Si \) on \(X \x \R \) given by \( \Si (dx, dE) : = \s_x(dE) m(dx) .\)
We want to show that the measures \( m\x Leb.\) and \(\Si\) are mutually singular. We exhibit two disjoint sets of full measure for respectively  \( m\x Leb.\) and \(\Si.\)

Fix $x \in X$. A sequence \(\{u_n\}_{n\in \Z} \in \C^\Z\) satisfies \( (H_x u)_n = Eu_n \) iff for all \(n\in \Z,\)
\begin{equation} \label{eigenvalue} \begin{pmatrix}  u_{n+1} \\   u_n \end{pmatrix} \;=\;  \begin{pmatrix}  E-V_n & -1 \\ 1 & 0 \end{pmatrix} \; \begin{pmatrix}  u_{n} \\  u_{n-1} \end{pmatrix}. \end{equation}

On the one hand, by Furstenberg theorem \ref{F}, as soon as $\tau $ is not supported on a point, the exponent \( \g (E)\) is positive. Then, by Oseledets theorem \ref{oseledets2}, for \(m\)-a.e. \(x \in X,\) a solution of (\ref{eigenvalue}) satisfies 
\[ \lim\limits_{n \to +\infty} \frac{1}{n} \log \max \{| u_n |, | u_{-n} |\} \; =\; \g (E) .\] 

On the other hand, the spectral theorem yields an unitary  isomorphism $U$ of \( \ell ^2(\Z) \) with a Hilbert integral \( \int _\R^\oplus V_E \, d\s _x (E) \) that conjugates \(H_x\) with the multiplication by $E$. Since \( H_x^n \de _0, n \in \Z \) and  \( H_x^m \de _1 , m\in \Z ,\) generate $\ell ^2 (\Z),$ the dimension of the Hilbert spaces $V_E$ is either   $1$ or $2$.  Assume \( \Dim V_E = 2 \; \s_x \)-a.e.(the case \( \Dim V_E = 1 \; \s_x \)-a.e is similar). Choose \( a(E), b(E) \in \ell ^2 (\Z) \) such that, for \(\s_x\)-a.e.$E, \; a(E),b(E) $ form an orthonormal basis of $V_E $ and define \( u_n (E), v_n (E) \) by  \[ U(\de _n) \;=\; u_n(E) a(E) + v_n (E) b(E) .\]
We have, for all $n\in \Z,$  \( \int (|u_n(E)|^2 + |v_n(E) |^2) \, d\s_x (E) = 1.\) Moreover, by applying \( UH_xU^{-1} \), we see that the sequences \( \{u_n(E) \}_{n \in \Z} , \{ v_n(E)\}_{n \in \Z} \) are solutions of (\ref{eigenvalue}) for $\s _x $-a.e.$E.$ Therefore, for \( \Si\)-a.e.\((x,E)\), there exists a  solution of (\ref{eigenvalue}) with  \( \sum _{n\in \Z} \frac{| u_n|^2 }{n^2}  < + \infty .\) \end{proof}

It was shown  later  that \( \s_x \) is a discrete measure for $m$-a.e. $x$ (see \cite{CKM87}).

A wide extension of Furstenberg theorem \ref{F} was suggested by the more general one-dimensional Anderson model.
Consider a topological dynamical system \((X,\A,m, f),\) where $X$ is a compact metric   space,
$\A$ the Borel $\s$-algebra, $f$ a homeomorphism and $m$ an ergodic $f$-invariant measure. . 

Let \(V: X \to \R \) be a continuous  function on $X$. The  Schr\"odinger operator at \( x \in X\) is the self-adjoint operator \(H_x\) of \( \ell ^2 (\Z) \) given for \( u = \{u_n \}_{n \in \Z} \in \ell ^2 (\Z), \)  for all \( n \in \Z,\) by 
\[\left( H_x (u)\right)_n \;:= \; u_{n+1} + u_{n-1} + V(f^n x)  u_n.\]
Fix \(E \in \R.\) Consider the field of matrices \( B_E (x) \;:= \)  \(\begin{pmatrix}  E-V(x) & -1 \\ 1 & 0 \end{pmatrix}\) and the exponent 
\[ \g (x,E ) \; := \; \lim\limits _{n \to \infty } \frac{1}{n} \log \| B_E(f^{n-1}x) \ldots B_E(x) \| .\]
By ergodicity, there is a constant \(\g(E)\) such that  \( \g(x,E) = \g(E) \)  for $m$-a.e. $x.$

The independent case corresponds to \( X = ({\textrm {supp.}} \tau)^{\otimes \Z}, m = \tau ^{\otimes \Z}, \; f \)  the shift transformation on $X$ and \( V(x) = x_0.\) 
\begin{definition} Let $(X,\A,m;f)$ be as above. A measurable function $f:X \to \R$ is called deterministic if  $V$ is measurable with respect to the $\s$-algebra generated by \( V(f^n x), n >0 .\)\end{definition}

\begin{theorem} [\cite{S83}] \label{kotani} Assume that the Lebesgue measure of \( \{ E \in \R: \g(E) =0 \} \) is positive. Then \( V \) is deterministic.\end{theorem}

S.~Kotani (\cite{K84}) had a similar result for the continuous Schr\"odinger equation. B.~Simon gave a direct proof for the discrete case. Of course, in the independent case, \(V\) is deterministic 
only if the measure $\tau $ is supported on a single point $E_0.$ Then, for all $E \in \R,$ the  matrix  {\scriptsize{\(\begin{pmatrix}  E-E_0 & -1 \\ 1 & 0 \end{pmatrix}\) }} preserves a probability measure on \( \PP^1. \) 

More generally, if the function $V$ is not deterministic, then \( \g (E) >0 \; Leb.\)-a.e. and the above Pastur argument applies:
\begin{corollary}  If the function $V$ is not deterministic, then for \(m\)-a.e. \(x\in X,\) the spectral measure \( \s_x \) is singular with respect to the Lebesgue measure on $\R$. \end{corollary}

\subsection{Random homeomorphisms  in dimension 1}

We now describe  a non-linear version of section 2.1. Let  $\mu$ be a probability measure on \( {\textrm {Hom}} (\SSSS^1)\).  Consider \( (\ov X, \ov \A, \ov m) \)   the product space \(\ov X :=  ( {\textrm {Hom}}(\SSSS^1))^{\otimes \N },\) with the $\s$-algebra \( \ov \A\) generated by the coordinates \( \{x_n \}_{n\ge 0}\) and the product measure \( \ov m := \mu ^{\otimes \N} .\) For \(x = \{x_n \}_{n\ge 0} \in \ov X, \, n\ge 1 \) set \[ \Phi_n (x):= x_{n-1} \ldots x_0.\] 

For \( \nu \) stationary probability measure on $\SSSS ^1,$ we can define  \( (\ov \EE,  \ov M)  :=  (\ov X \times \SSSS^1, \ov m \times \nu )\)
 and the transformation   \( \ov F: \ov \EE \to \ov \EE ,\) defined by  \(  \ov F (x, \xi ) \;:= \; (\ov f (x), x_0 \xi)\), where $\ov f$ is the shift on the product space \( \ov X.\) 
As before, the measure $\ov M$ is $\ov F$-invariant, $\ov F$-ergodic if, and only if, the measure $\nu$ is extremal among stationary probability measures.
 It appears (see \cite{LJ87}, \cite{K93} for particular cases) that there is a strong  dichotomy: either there is a probability   measure which  is invariant under $\mu$-a.e.$T,$ or, for large $n$, the typical  iterate $\Phi _n$ is a North-South diffeomorphism of $\SSSS ^1$, with a finite number of sinks and sources. It was observed by A. Carverhill (\cite{C87}, P. Baxendale (\cite{B89}) and H. Crauel (\cite {C90}) that this phenomenon is very much related to theorem \ref{F}  in the case of diffeomorphisms. It was extended to homeomorphisms by D. Malicet
 (see \cite{M17}).

We define, for $(x,\xi ) \in \ov \EE,$
\[ \g(x,\xi) \;:=\; \liminf _{\eta \to \xi } \liminf _{n\to +\infty} -\frac{1}{n} \log |\Phi_n (\eta) -\Phi _n (\xi) | .\]
Clearly, we have $\g(x,\xi) \ge 0 $ for all $(x,\xi).$
\begin{theorem}[\cite{M17}]\label{malicet}  Let  $\nu$ be a $\mu$-stationary probability measure on \(\SSSS^1\). If $\g (x,\xi) = 0  \; \ov M $-a.e., then  $\nu$ is   invariant under $\mu$-a.e. $T $. \end{theorem}
The proof of theorem \ref{malicet} follows the scheme of the proof of theorem \ref{F}.
For a stationary probability measure $\nu,$ we define the Furstenberg entropy \[ \kappa (\mu, \nu) := \int  \log \frac {dT_\ast \nu }{d\nu } (T\xi ) \, d\nu (\xi) \, d\mu (T)\] (cf. definition \ref{Fentropy}). We still have \( \kappa (\mu, \nu ) \ge 0 ,\) with equality if, and only if, the measure $\nu $  is invariant under $\mu$-a.e. $T .$ We have another basic inequality
\begin{proposition}\label{basicinequality2} Let $\nu $ be a stationary probability measure. Then,
\[\kappa (\mu, \nu ) \; \le \; \int \g (x,\xi) \, d \ov M (x, \xi).\]
 \end{proposition}
 \begin{proof} 
We may assume that the measure $\nu $ is extremal. Then, by ergodicity, we may assume that the measure $\nu $ is continuous since, otherwise, it is discrete and uniform on a finite set. Then $\kappa (\mu, \nu ) =0$ and the proposition holds. 
 
Recall the definition of $\g$. We can write, at $\ov M$-a.e. $(x,\xi),$
\begin{eqnarray*}  \g(x,\xi) &= &\liminf_{\eta \to \xi } \liminf _{n\to +\infty} -\frac{1}{n} \log |\Phi_n (\eta) -\Phi _n (\xi) | \\ &\ge& \liminf _{I\ni \xi,|I| \to 0 } \liminf _{n\to +\infty} -\frac{1}{n} \log |\Phi _n (x) (I))| \\&\ge& \liminf _{I\ni \xi,|I| \to 0 } \liminf_{n\to +\infty} -\frac{1}{n} \log \nu (\Phi _n (x) (I)) -  \limsup _{n\to +\infty} \frac{1}{n} \log ^+ Q(\Phi _n (x)\xi) \\&\ge& \liminf _{I\ni \xi,|I| \to 0 } \liminf_{n\to +\infty} -\frac{1}{n} \log \nu (\Phi _n (x) (I)),   \end{eqnarray*}
where $Q: \SSSS ^1 \to \R_+ $ is given by $Q( \eta ):=   \sup_{I \ni \eta} \frac {|I|}{\nu (I)} .$ The function  $\log ^+ Q  $ is $\nu$-integrable by corollary \ref{density} and thus, for $\ov M$-a.e. $(x, \xi), \lim _n  \frac{1}{n} \log ^+ Q(\Phi _n (x)\xi) =0.$

On the other hand, since $\nu$ is continuous, we may define,  for any $\de >0$ and $(x,\xi ) \in \ov \EE,$ 
\[ J_\de (x, \xi) \; := \; \sup \{ \frac{\nu (x_0 I)}{\nu (I)} : I \ni \xi, |I| < \de \} .\]
 By proposition \ref{densities}, $\kappa(\mu, \nu ) = \lim\limits _{\de \to 0}-  \int \log J_\de \, d\ov M .$  Assume $\kappa (\mu, \nu ) < + \infty .$ Then, for every fixed $\e >0,$ there is $\de >0$ and a $\ov M$-a.e. positive and finite function $C_\e (x,\xi) $ such that, for all $n \ge 0,$ 
\begin{equation}\label{controlJ} \Pi _{j=0}^{n-1} J_\de (\ov F^j (x,\xi) ) \;\le  \; C_\e (x, \xi ) e^{-n (\kappa -\e)} .\end{equation}

There is nothing to prove if $\kappa (\mu, \nu ) =0$. Otherwise, for any $\e < \kappa (\mu ,\nu),$ choose $\de >0 $ and, for $M$-a.e. $(x,\xi), \; C_\e(x, \xi) \ge 1$ such that equation (\ref{controlJ}) holds $\ov M$-a.e..
\begin{claim}\label{induction} If $I$ is an interval containing $\xi $ such that $\nu (I) \le \de / C_\e(x,\xi), $ then, for all $n\ge 0 , \, \ov M$-a.e. $(x,\xi) \in \ov \EE,$  \[ \nu (\Phi _n (x) (I))\; \le \;  C_\e (x,\xi) \nu (I) e^{-n(\kappa -\e)}.\] \end{claim}
\begin{proof}We prove the inequality by induction on $n$. It holds for $n= 0$ since $C_\e \ge 1.$ Assume that the inequality holds for $0 \le k < n.$ Then in particular, $\nu (\Phi _k (x) (I)) \le \de $ for $0 \le k < n$ and therefore, since $\Phi _k (x)(I) \supset \Phi _k (x)(\xi),$  \[ \frac {\nu (\phi _{k+1}(x) (I))} {\nu (\phi _{k}(x) (I))}\;  \le \; J_\de ( \ov F^k (x,\xi)).\]
The claim follows from (\ref{controlJ}) since the product telescopes.
\end{proof}
Reporting the estimate of  claim \ref{induction} in the above expression for $\g(x,\xi )$ at $\ov M$-a.e. $(x,\xi),$ we obtain $\g \ge \kappa -\e $  for all $\nu, \e  $ with $\nu $ extremal and  $\kappa (\mu , \nu ) >\e.$ If $\kappa (\mu , \nu )  = +\infty,$ then a similar argument yields $\g = + \infty $ and proposition \ref{basicinequality2} is proven in all cases.
\end{proof}
\begin{proof}({\it{Proof of theorem \ref{malicet})}} We may assume that the measure $\nu$ is extremal. If $\g^- = 0,$  then, by proposition \ref{basicinequality2}, \( \kappa(\mu, \nu ) =  0 .\)  We conclude that  $\nu $ is  invariant by the analog of proposition \ref{invariance1}. \end{proof}

 See also \cite{DKN07}. Random pseudo-groups of homeomorphisms occur naturally when studying the transversal properties of foliated Brownian motion. Analogs of theorem \ref{malicet} in the piecewise smooth case have been proven independently by C. Bonatti and X. Gomez-Mont (\cite{BGM01}) and by B.  Deroin and V. Klepstyn (\cite{DK07}) for codimension 1 foliations with hyperbolic leaves. 
 
 Observe that theorem \ref{F} follows from theorem \ref{malicet} by associating to a matrix $A \in \SL(2,\R)$ the continuous mapping $T \in {\textrm {Hom}} (\PP^1(\R))$ defined by $T[\xi] = [A\xi].$ For $\nu $ an extremal  stationary probability measure on $\PP^1 (\R),$ the associated contraction coefficient $\g$ is twice the Lyapunov exponent of the random walk (see \cite[section 3.1]{M17}).

\subsection{Compact extension of hyperbolic systems}\label{PHcompact}

\begin{definition}\label{PH} Let $Y$ be a compact connected Riemannian manifold. The   \(C^1\) diffeomorphism \( \vf \) of  $Y$ is called {\bf{partially hyperbolic }} if there exist  a continuous decomposition \( y \mapsto E^u_y \oplus E^c_y \oplus E_y^s \) of $T_yY$  and a Riemannian metric  such that 
\begin{itemize}\item  for all \(v^s \in E^s_y, v^c \in E^c_y \) and  \(   v^u \in E^u_y , \|v^\ast \| = 1  \) for \( \ast = s,c,u,\) we have \[ \|D_y \vf (v^s) \| <   \|D_y \vf (v^c) \| <   \| D_y \vf ^{-1}(v^u) \| \]
\item and  \( \|D_y \vf |_{E^s} \| <  1 , \;\;  \| D_y \vf ^{-1}|_{E^u} \| < 1.\)
\end{itemize} \end{definition}

Let $\vf $ be a partially hyperbolic diffeomorphism of the compact manifold $Y$. By compactness,   the following stable manifold theorem applies to $f= \vf, V= E^s$ and to $f =\vf ^{-1}, V= E^u.$

\begin{proposition}[\cite{A67}]\label{stablemanif} Assume $X$ is a compact Riemannian manifold with Riemannian distance $d$, $f : X\to X $ a partially hyperbolic $C^{1+\alpha}$ diffeomorphism and $m$ an ergodic $f$-invariant probability measure. Let \( \|D_y f|_{E^s} \| < \rho  < \|(D_y f |_{E^c} )^{-1}\|^{-1} \le 1.\) Then, 
\begin{itemize} \item for $x \in X$, 
\[ W^\rho_x := \{ y \in X, \limsup _{n \to +\infty }\frac{1}{n} \log d(f^ny, f^n x ) \le  \log \rho \} \] is an embedded space \( \R^d\)  with \( T_x W^\rho_x = V(x) \) for all $x \in X.$ Moreover,
\item the \( W^\rho _x, x \in X\) form a H\"older continuous lamination $\W^\rho.$ In particular, for \( y\in W^\rho _x, \) there exists a locally H\"older continuous one-to-one holonomy map \( h^\rho_{x,y} \) from a neighborhood of   $x$ in any small transversal to $\W^\rho$ to a neighborhood of   $y$ in any small transversal to $\W^\rho$.
\end{itemize} \end{proposition}

We denote \( \W^u\) (resp. \( \W^s\)) the unstable (resp. stable) foliations associated to the fields  \( E^u\) (resp. \(E^s\)), $W^u_y, W^s_y, \, y \in Y,$ the individual leaf at $y$, $h^u, h^s $ the corresponding local holonomy maps, defined in local charts of the foliations. 

\begin{definition}\label{suproduct} Assume that $(Y,\vf) $ is partially hyperbolic with  $\dim E^c =0.$ A measure $\nu$ is called a {\bf{$(s,u) $-product }} if, in a common local chart for both foliations,
\begin{itemize} \item the conditional $\nu ^u_y $ of the measure $\nu $ with respect to the partition in local  \( W^u \)-leaves are $h^s$-invariant or, equivalently,
 \item the conditional $\nu ^s_y $ of the measure $\nu $ with respect to the partition in local  \( W^s \)-leaves are $h^u$-invariant.
 \end{itemize}\end{definition}

 The equivalence of the two statements in definition \ref{suproduct} is a direct application of Fubini Theorem. Famously, for $(Y,\vf) $ hyperbolic (i.e. partially hyperbolic with  $\dim E^c =0$), G. Margulis (\cite{M69}) constructs an invariant probability measure which is a $(s,u)$-product and R. Bowen  (\cite{B75}) shows that Margulis measure  is the unique measure of maximal entropy for $\vf.$ Later, R. Bowen and B. Marcus (\cite{BM77}) show that Margulis measure is the  only $(s,u)$-product probability measure. Extending these discussions to general partially hyperbolic systems is natural when the system is {\it{dynamically coherent}}.
 
 \begin{definition}\label{DC}  A partially hyperbolic system $(Y,\vf) $ is called {\bf{dynamically coherent }} if there exist continuous foliations $
 \W^{cs}, \W^{cu} $ such that,  at every $y \in Y$, \( T_y W^{cs}_y = E^s_y \oplus E^c_y, \,  T_y W^{cu}_y = E^c_y \oplus E^u_y .\)
\end{definition}
 
 Partially hyperbolic systems with compact center leaves, perturbations of time-one maps of geodesic flows in negative curvature, non hyperbolic ergodic automorphisms of the torus are dynamically coherent. Not  all partially hyperbolic systems are dynamically coherent, even in dimension 3 (see \cite{RHRHU}, \cite{BGHP}). If $(Y,\vf)$ is partially hyperbolic and dynamically coherent, then the lamination \(\ov  \W^c \) with leaves the connected components of 
 \[ \ov W^c_y \; :=\;  W^{cs}_y\cap W^{cu}_y \]  is H\"older continuous and  integrates the direction $E^c_y.$  The lamination $\ov \W^c$ is globally invariant under $\vf. $
 Let $h^c $ be the local holonomy mappings between transversals to $\ov \W ^c.$
 The mappings $h^c $ preserve the local $
 \W^{cs}, \W^{cu} $ leaves. We extend the definition \ref{suproduct} as follows:
 
 \begin{definition}\label{scuproduct} Assume that $(Y,\vf) $ is partially hyperbolic and dynamically coherent. A measure $\nu$ is called a {\bf {partial $(s,c,u) $-product}}  if
\begin{itemize} \item the conditionals $\nu ^{cs}_y $ of the measure $\nu $ with respect to the partition in local  \( W^{cs} \)-leaves are $(s,c)$-products and
 \item the conditionals $\nu ^{cu}_y $ of the measure $\nu $ with respect to the partition in local  \( W^{cu} \)-leaves are $ (u,c)$-products.
 \end{itemize}\end{definition}
 
 \begin{definition} Let $(Y,\vf) $ be a partially hyperbolic, dynamically coherent system with compact central leaves, $(Z, \psi) $ the quotient  system made of the space of leaves.  The system $(Y, \vf)$ is called a {\bf{compact extension of hyperbolic system}}  if the system $(Z,\psi) $ is a Smale space in the sense of  \cite[chapter 7]{R04}. \end{definition}

 Let $(Y,\vf) $ be a  partially hyperbolic system,   $\nu$ be a $\vf$-invariant ergodic probability measure on $Y.$  We can define the upper and lower  center Lyapunov exponents by the following $\nu$-a.e. limits
 \[ \g^+_c \;:=\; \lim\limits _{n \to + \infty } \frac{1}{n} \log \| D_y \vf ^n |_{E^c}\| , \quad \quad \g^-_c \;:=\; \lim\limits _{n \to + \infty }- \frac{1}{n} \log \| D_y \vf ^{-n}|_{E^c}\|.\]

  A. Tahzibi and J. Yang (see \cite{TY19}) observe the following  IP. 
  \begin{theorem}[\cite{TY19}]\label{TY}  Let $(Y,\vf) $ be a $C^{1+\alpha}$ compact extension of hyperbolic system. Let $\nu$ be an invariant ergodic probability measure and assume that \( \g^+_c = \g^-_c = 0.\) Then, $\nu $ is a partial $(s,c,u)$-product measure. \end{theorem}
 \begin{proof}\footnote{We sketch  the proof from \cite{TY19}, underlining the similarities with the proof of theorems \ref{F} and \ref{malicet}. We hope that it makes the later extensions more clear.}
 Let $Z$ be the space of compact leaves, endowed with a quotient metric, $\pi : Y \to Z$  the quotient map. By hypothesis, the action of $\vf$ on $Z$ is a H\"older continuous hyperbolic map $\psi$. It has a pair of  continuous laminations 
 $\ZZ^s, \ZZ^u $ with local leaves $Z_{loc}^s(x), Z_{loc}^u (x)$ such that   for all $y \in Y$ and some $\rho < 1, \e >0$,
 \[ Z_{loc}^s (x) := \{ y \in X,  d(\psi^ny, \psi ^n x ) \le  \e  \rho ^n {\textrm { for all }} n \ge 0\}, \]
 \[ Z_{loc}^u (x) := \{ y \in X,  d(\psi ^{-n}y, \psi^{-n} x ) \le  \e \rho ^n {\textrm { for all }} n \ge 0\}\} .\]
By \cite[ chapter 7.5]{R04}, for any $\e >0$ there exists a Markov partition $\RR = \{R_j, 1\le j \le J\}, $ such that the rectangles  $R_j$ have diameter smaller than $\e$. 
 We can also arrange that the boundaries of the rectangles $R_j$  have  0 $ \pi _\ast \nu $-measure.
 
 We show that for $\nu $-a.e. $y$, the conditionals  of the measure $\nu $ with respect to the partition in local  \( W^{cu} \)-leaves are $ (u,c)$-products. The proof is the same for the conditionals associated to the local central stable leaves. We set, for $z \in Z, y \in Y,$
 \[ \zeta ^u (z) : = \cap _{n \ge 0} \psi ^{-n} \RR (\psi ^n (z) ),  \;  \xi ^{cu} (y) := \pi ^{-1} (\zeta ^u (\pi (y)) \; {\textrm {and} } \; \xi ^{u} (y) := \xi^{cu} (y) \cap W^u_{loc} (y) .\]
 The elements of the  partitions \( \zeta ^u, \xi ^u\) and \( \xi ^{cu}\) have open interior  and negligible boundaries in the local leaves, respectively \( Z_{loc}^u, W_{loc}^u\) and $W_{loc}^{cu}.$ The partitions $ \xi ^{cu}$ and $\xi ^u $ are measurable, let $ \nu _y^u, \nu _y^{cu} $ be the associated conditional measures. For  a measurable partition $\xi $ with conditional measures $\nu _y, y \in Y$ and a finite partition $\CC = \{C_\ell, 1\le \ell \le L \}, $ write the  conditional entropy $H(\CC |\xi ) $ as:
 \begin{equation}\label{conditionalentropy} H(\CC|\xi) \;:=\;  \int -\sum _{1\le \ell \le L}  \nu _y  (C_{\ell})\log  \nu _y  (C_{\ell}) \, d\nu (y) .\end{equation}
 The following proposition is classical,  we state it as the analog of proposition \ref{invariance1}  in our setting
 \begin{proposition}\label{invariance4} Let $\CC_k = C_{k,\ell}, 1\le \ell \le L_k $ be  a  finite partition of $Z$. We have:
\[ H( \pi ^{-1} \CC _k | \xi ^u )\; \le H(\pi ^{-1} \CC _k | \xi ^{cu}) .\]
 We have equality for finer and finer partitions $\CC _k $ such that, for $\nu $-a.e. \(y\in Y, \; \CC_k, k \to \infty, \) generate the Borel $\s$-algebra of $\zeta^u (\pi (y))  $ if, and only if, for $\nu$-a.e. $y\in Y,$  the conditionals $\nu ^{cu}_y $  are $ (u,c)$-products. \end{proposition}
 \begin{proof}We may write, for all $k, \ell, 1\le \ell  \le L_k,$ $\nu$-a.e. $y  \in Y,$ \[ \nu ^{cu}_y (\pi ^{-1}C_{k,\ell} )\;=\; \int \nu ^{u}_w (\pi ^{-1}C_{k,\ell}) \,d \nu ^{cu}_y (w).\]
 Using Jensen inequality with the  function $ t\mapsto -t\log t ,$ we obtain, for all $k, \ell, 1\le \ell  \le L_k,$ $\nu$-a.e. $y \in Y,$
 \[  -\int \nu ^{u}_w (\pi ^{-1}C_{k,\ell}) \log [\nu ^{u}_w (\pi ^{-1}C_{k,\ell}) ] \, d\nu ^{cu}_y (w)  \; \le \; - \nu ^{cu}_y (\pi ^{-1}C_{k,\ell} ) \log [\nu ^{cu}_y (\pi ^{-1}C_{k,\ell}) ] . \]
 For fixed $k$, the announced inequality follows by summing in $\ell , 1\le \ell \le L_k$ and integrating in $y$. In case of equality, we have, for all $ \ell, 1\le \ell  \le L_k,$ $\nu$-a.e. $y \in Y,$ and $\nu ^{cu}_y$-a.e. $w \in \xi ^{cu} (y) ,$  \(  \nu ^{u}_w (\pi ^{-1}C_{k,\ell}) =   \nu ^{cu}_y (\pi ^{-1}C_{k,\ell}) .\)
 
 So, in case of equality,  we have, for $\nu$-a.e. $y \in Y$ and for $\nu ^{cu}_y$-a.e. $w \in \xi ^{cu} (y) , \, \pi _\ast (\nu ^u_w) = \pi _\ast (\nu ^{cu} _y) $ on $\CC _k.$ As $k\to \infty $, we obtain that for  $\nu$-a.e. $y \in Y $ and $ \nu ^{cu}_y$-a.e. $w \in \xi ^{cu} (y) , \, \pi _\ast (\nu ^u_w) = \pi _\ast (\nu ^{cu} _y) $  on $\zeta ^u (\pi (y)).$  The proposition follows since the $\pi _\ast (\nu ^{cu}_y), y \in Y,$ are the conditional measures associated with the partition $\zeta^u$.
\end{proof} 
In order to prove theorem \ref{TY}, we  apply  proposition  \ref{basicinequality3} to the partitions \( \CC _k \) defined by $\CC _k (z)  :=  \cap _{n = 1}^k  \psi ^{n} \RR (\psi ^{-n} (z) ).$ Since \( \zeta ^u (z) : = \cap _{n \ge 0} \psi ^{-n} \RR (\psi ^n (z) ),  \) the $\CC _k $ form an increasing family that generates the Borel $\s$-algebra of $\zeta ^u (z) $ for all $z \in \ZZ.$

Observe that, by construction, the partition $\xi ^u$ (resp. $\xi ^{cu} $) 
is an increasing  partition made, up to measure 0, of  open subsets of local unstable (resp. central unstable) leaves. Furthermore, we have (see proposition \ref{entropyrohlin}) \[H(\pi^{-1} (\CC_k) |\xi ^u ) = H (\vf ^{-k} \xi ^u |\xi ^u), \quad   H(\pi^{-1} (\CC_k) |\xi ^{cu} ) = H (\vf ^{-k} \xi ^{cu} |\xi ^{cu}).\]
Theorem \ref{TY} follows from proposition \ref{basicinequality3}. \end{proof}

\section{Invariance Principle}\label{TheoremIP}

\subsection{Basic setting for the IP}

Consider an  measurable  ergodic dynamical system \((X,\A,m, f),\) where $(X,\A, m) $ is a Lebesgue space (see section \ref{Rokhlin}),
$f$ a bimeasurable bijection of $X$ preserving  $m$. 
Let \( \Pi : \EE \to X \) be a measurable  fiber bundle over  $X$ with fibers \(\EE_x \)  compact Riemannian manifolds of bounded  dimensions and diameters. We also assume that the $C^2$ norms of the  metrics on the fibers $\EE _x$ are uniformly bounded.

Let \( F: \EE \to \EE \) be a measurable   bundle of $C^2$-diffeomorphisms   \( F_x : \EE _x \to \EE _{f(x)} . \) We assume that  \( \|D_\xi F_x\|, \;  \|D_{\xi} (F_{x})^{-1}\| \) and \( \| \Hess_\xi (F_x)\|\)  are uniformly bounded. Write
\[ F^n = 
\{ F^{(n)}_x, x \in X \}, \; {\textrm {where}} \; F^{(n)}_x := F_{f^{n-1} x } \circ \ldots \circ  F_x .\]
 
Let $M$ be an ergodic $F$-invariant probability  measure on $\EE$ such that \( \Pi _\ast M = m\) and write \( M = \int M_x \, dm (x) \) the disintegration of $M$ for $\Pi.$ Observe that by invariance, for $m$-a.e. $x,$ \begin{equation}\label{invariance} (F_x)_\ast  M_x \;= \; M_{fx} .\end{equation}

By the subadditive ergodic theorem, there are numbers \(\g ^+, \g ^-,\) with \( \g ^+ \ge \g ^-, \) such that,  at $M$-a.e. \((x,\xi), \xi \in \EE_x,\) 
\[ \g^+  \,:=\, \lim\limits _{n \to \infty } \frac{1}{n} \log \| D_\xi F^{(n)}_x  \|,  \quad  \g ^- :=\, -\lim\limits _{n \to \infty } \frac{1}{n} \log \| (D_\xi F^{(n)}_x )^{-1} \| .\] Call $\g^+, \g^- $ the associated exponents.

Let \((X,\A,m, f)\)  be as above. We say that a sub-$\s$-algebras \( \B\) is finer  than a sub-$\s$-algebra \( \CC \subset \A \) (noted \( \CC \subset \B\)) if for any set \( C\in \CC,\) there is a set \( B \in \B \) with \( m(B\De C) = 0.\) We say that the sub-$\s$-algebra \(\B \subset \A\) is non-increasing if \( f^{-1} \B \subset B,\) generating  if the $\s$-algebra generated by all \( \{f^n \B, n\in \Z\} \) is finer than  $\A.$

\begin{theorem} [\cite{AV10}]\label{IP} 
With the above notations\footnote{In particular, for the sake of the proof we give, we assume that  \( \|D_\xi F_x\|, \;  \|D_{\xi} (F_{x})^{-1}\| \) and \( \|\Hess_\xi (F_x)\|\)  are uniformly bounded. See subsection \ref{regularity}.}, assume there is a generating non-increasing sub-$\s$-algebra \(\B\), that the mappings \(x\mapsto \EE_x,  x \mapsto F_x \) are \(\B\)-measurable, and that \( \g ^-  \ge 0.\)
Then, the mapping \( x \mapsto M_x \) is $\B$-measurable. \end{theorem}

\subsection{Proof of theorem \ref{IP}} 
By section \ref{Rokhlin}, given a $\s$-algebra $\B$ satisfying the hypotheses of theorem \ref{IP},  there is an essentially unique Lebesgue  space \(( \ov X, \ov \A , \ov m) \) and a projection $ \pi : X \to \ov X $ such that \( \pi ^{-1} \ov A = \B\) and \( \pi _\ast m = \ov m.\) Since \( \B \) is non-increasing, there is a $\ov m$-preserving mapping (non-invertible in general)  $\ov f$ such that \( \pi (fx) = \ov f (\pi x ) .\)

The measurability hypotheses ensure that there exist  a fiber bundle \( \ov \Pi: \ov \EE \to \ov X \) with fibers compact manifolds, a fibered projection \( \ov \pi : \EE  \to \ov \EE \) and a quotient mapping \( \ov F\) on \(\ov \EE\) so that \( \ov F \circ \ov \pi = \ov \pi  \circ  F \) and \( \ov  \Pi \circ  \ov F =  \ov  f \circ \ov  \Pi .\) The measure \( \ov M := \ov \pi_\ast M \) is a \( \ov F\)-invariant ergodic probability  measure on \( \ov \EE.\)  Let \( \ov M = \int \ov M_{\ov x }\, d\ov m (\ov x) \) be a disintegration of \( \ov M \) with respect to \(\ov  \Pi .\)

We work with the family of measures \( \ov M_{\ov x }, \) defined for \(\ov m \)-a.e. \(\ov x \in \ov X.\) We have to prove that, for $m$-a.e.$x \in X, \;   M_x = \ov M_{\pi (x)}.$  We define the entropy \(\kappa: \)
\begin{definition}\[ \kappa :=  \int \left( \int  \log \frac {d(\ov F_{\ov x})_\ast \ov M_{\ov x}}{d\ov M_{\ov f \ov x }} (\ov F_{\ov x } \ov \xi) \, d \ov M_{\ov x } \right) \, d\ov m_{\ov x }.\]\end{definition}

The first step is analogous to proposition \ref{invariance1}, with the same proof:

\begin{proposition}\label{invariance2} We have \( \kappa \ge 0,\)  with equality only if, for  \(  \ov m\)-a.e. \( \ov x \in \ov X, \) \[ (\ov F_{\ov x})_\ast \ov M_{\ov x} = \ov M_{\ov f \ov x} .\]\end{proposition}

The second step is the observation that the family \( M_x, x\in X\) can be recovered from the family of measures \( \ov M_{\ov x }. \)

\begin{proposition}\label{martingale}  For \(m\)-a.e. $x$,  \( M_x = \lim\limits _{n\to +\infty} ((\ov F^{(n)}_{\ov  x})^{-1})_\ast \ov M_ {\ov {f}^n \pi x} .\) \end{proposition}
\begin{proof} For $n\ge 0,$ let \(X_n\) be the space associated to \( f^{-n} \B \supset \B, \; \pi _n \) the associated projection. 
There are a  fiber bundle \(\Pi _n : \EE _n \to X_n \)  and a fibered projection \(\ov \pi _n : \EE \to \EE _n \). Since  the $\s$-algebra $\B $ is non-increasing and $ x \mapsto \EE_x $ is $\B$ measurable, the fibers of $\EE _n$ are identified with the fibers $\EE _{\pi _n(x)}, x \in X$ for all $n\ge 0.$
Since \(\B\) is generating, by the martingale theorem, for $m$-a.e. $x$, the disintegrations  \( M_x, x\in X\) are given by the weak*-limits of the disintegrations $M^n_{\pi _n (x) } $ associated to $\Pi _n.$ By invariance (\ref{invariance}), \[ M^n_{\pi _n x}  = ((F^{(n)}_{\pi _nx})^{-1})_\ast  M^0_{f^n( \pi _n x)}.\]
 With our notations, $\pi _0 (x) = \ov x, f^n (\pi _n x) = \ov{f^n x} = \ov f ^n \ov x $ and $M^0_{\ov x} = \ov M_{\ov x},$ so that 
\[ M_x \,=\, \lim\limits _{n \to \infty } ((F^{(n)}_{\pi _n x})^{-1})_\ast \ov M_{\ov f^n \ov x}.\]
Since $x \mapsto F_x $ is $\B$-measurable, $F^{(n)}_{\pi _n x} = F^{(n)}_{\pi _0 x} = \ov F^{(n)}_{\ov x}.$  \end{proof}

The third step is a   basic inequality:

\begin{proposition}\label{basicinequality4} 
Let $d$ be the maximal dimension of the fibers $\EE_x.$ Then, \[ \kappa \le - d \min \{ \g^-, 0\} .\]
 \end{proposition}

\begin{proof}
Assume first that $\g ^- >0.$ Then, in Pesin charts (see section \ref{pesin}), for $\ov m$-a.e. $\ov x$,  the mappings $F_x$ are expanding around 
$\ov M_{\ov x}$-almost every  point  $\xi \in \EE _x $ and thus $\ov M_{\ov x}$-almost every  point  $\xi \in \EE _x $ has positive measure. The measure $\ov M_{\ov x} $ is therefore discrete for $\ov m$-a.e. $\ov x.$ The individual masses are uniform on each $\EE _x$  and  constant by ergodicity. Therefore, there exists an integer $k \ge 1$ so that, for $\ov M$-a.e. $(\ov x,\xi), \; \ov M_{\ov x} (\{\xi\}) = 1/k .$  So, if $\g^- >0, $ we have $\kappa =0 = - d \min \{ \g^-, 0\} .$ 
 
Assume now that  $\g ^- \le 0. $
By the subadditive ergodic theorem, \[ -\g^- \;=\; \inf _n \frac {1}{n} \int \log \| (D_\xi F^{(n)}_{\ov x })^{-1} \|\, d\ov M (\ov x,\xi) \] 
  and given $\e >0,$  one can find $L,$ depending on $\e,$  such that  \[  \int \log \| (D_\xi F^{(L)}_{\ov x} )^{-1}\| \, d\ov M (\ov x,\xi) \le -L (\g ^- -\e/2 ) .\]
The probability measure $\ov M$ might not be ergodic for $F^{(L)}$. Let $\ov M = \frac{1}{I} \sum _{i \in I} \ov M_i$ be an ergodic decomposition of $\ov M$ for $F^{(L)}.$ 
Fix $\e >0$ and observe that,  for $i \in I,$  \( \int \log \| (D_\xi F^{(L)}_{\ov x} )^{-1}\| \, d\ov M_i (\ov x,\xi) \le  -L (\g^- -\e/2 )\) (by the ergodic theorem, the integral \( \int \log \| (D_\xi F^{(L)}_{\ov x} )^{-1}\| \, d\ov M_i (\ov x,\xi) \) does not depend on $i$ and  is at most   $-L(\g^- -\e/2)$ for all $i\in I$.)

For $i \in I$ and  for $\ov m$-a.e. $\ov x$, let  $\ov M_{i,\ov x}$ be a  disintegration of $\ov M _i.$  Write, for $\ov x\in \ov X, \xi \in \EE_{\ov x},i \in I$ 
\[ g_i(\ov x, \xi) \;:= \; \frac {d((\ov F^{(L)}_{\ov x})^{-1})_\ast \ov M_{i, \ov f^{L}\ov x} }{d\ov M_{i, \ov x}}(\xi), \quad \kappa _i \;:= \; - \int \log g_i (\ov x,\xi) \, dM_i (\ov x,\xi) .\]
For  $N$ large,  set \( g_i^N (\ov x,\xi) := \max \{g_i(\ov x, \xi), e^{-N}\} .\) The function $-\log g_i^N $ is $\ov M_i$-integrable for all $i \in I,$ $ -\log g_i \ge - \log g_i^N \)  and, by monotone limit,  \begin{equation}\label{infinite} L  \kappa \; = \; \frac{1}{I} \sum _{i \in I} \kappa _i\;=\;  \frac{1}{I} \sum _{i \in I} \lim\limits _{N \to \infty } - \int \log g_i^N (\ov x,\xi) \, d\ov M_i(\ov x,\xi).\end{equation}

 \begin{lemma}\label{telescope2}  Fix $N, \e, L$ and $i \in I$ as above and assume that  there exist $\beta >0$ and   $J \in \N$ such that, for $\ov M_i$-a.e. $(\ov x,\xi) \in \ov \EE,$ there exists a sequence  $n_k(\ov x,\xi), k>0,$ of  upper density at least $\beta,$ and for each $k$, a sequence   $q (j,n_k,\ov x,\xi) ,0 \le j\le n_k$ of measurable  positive functions  on $\ov \EE$   with the property  that  for all $k$, for $\ov M_i$-a.e.$(\ov x,\xi)$,
 \begin{eqnarray}\label{sizeofq3} & &  q (n_k, n_k, \ov x, \xi ) =  e^{(\g^- -2\e)L n_k} , \\ \label{sizeofq4} & & q (j, n_k ,\ov x, \xi ) \le  e^{(\g^- -2 \e)L  j} , \; {\textrm { for all $j\le n_k$ and }}\\
\label{balls2}&. & F^{(L)}_{\ov f^{Lj} \ov x} B(F^{(Lj)}_{\ov x} \xi, q(j,n_k,\ov x,\xi ))) 
\; \supset  B(F^{(L(j+1))}_{\ov x} \xi, q(j+1,n_k, \ov x,\xi )))\end{eqnarray}
for all $J\le j < n_i.$ 

Then, $ - \int \log g_i^N (\ov x,\xi) \, d\ov M_i(\ov x,\xi)\le dL(-\g^-  + 3\e).$ \end{lemma} 
\begin{proof}
Choose  $\de >0$ such that, setting 
\begin{eqnarray*}  g_{i,\de} (\ov x,\xi) &:= & \frac {\int _{B(\xi, \de)} g_i(\ov x,\eta) \, d\ov M_{i,\ov x} (\eta)}{\ov M_{i,\ov x} (B(\xi,\de))} =  \frac {\ov  M_{\i,ov f^{L}\ov x}  ( F^{(L)}_{\ov x} (B(\xi, \de)))}{\ov M_{i,\ov x}(B(\xi,\de))},\\
 g_{i,\de}^N (\ov x,\xi) &:= & \frac {\int _{B(\xi, \de)} g_{i,\de}^N(\ov x,\eta) \, d\ov M_{i,\ov x} (\eta)}{\ov M_{i,\ov x} (B(\xi,\de))} , \quad 
 g_{i,\de}^\ast := \sup _{\de ' \le \de } g_{i,\de '}^N,\end{eqnarray*}
we have $\displaystyle \int \log \frac {g_{i,\de}^\ast}{g_{i}^N} \, d\ov M_i  \le \e$ (see proposition \ref{densities}). Observe that, for $q\le \de ,$ \[-\log g_i^N \le -\log g_{i,q} +\log g_{i,q} -  \log g_{i,q}^N +\log g^\ast_{i,\de} - \log g_i^N  \; \le \; -\log g_{i,q} +\log g^\ast_{i,\de} - \log g_i^N .  \]

We may  thus  write, at $\ov M_i$-a.e.$(\ov x,\xi)$, for $J$ such that $e^{(\g^- -2\e)JL} <\de ,$
  \begin{eqnarray*} -\int \log g_i^N \, d \ov M_i  &\le & \liminf_{k \to \infty } -\frac{1}{n_k}\sum _{j=J}^{n_k-1} \log g_{i,q (j, n_k, \ov x, \xi ) }(\ov F^{jL} (\ov x,\xi)) \; +\; \e \\
     &\le & \liminf_{k \to \infty } -\frac{1}{n_k}\sum _{j=J}^{n_k-1}\log  \frac {\ov M_{i,\ov f^{L(j+1)} \ov x} (F^{(L)}_{f^{jL}\ov x} B(F^{(Lj)}_{\ov x} \xi  ,q(j,n_k, \ov x,\xi )))}{\ov M_{i,\ov f^{Lj} \ov x} B(F^{(Lj)}_{\ov x} \xi, q(j,n_k, \ov x,\xi )))} \; +\; \e \\
  &\le &  \liminf_{k \to \infty } -\frac{1}{n_k} \log \frac {\ov M_{i,\ov f^{Ln_k} \ov x} (B(F^{(Ln_k)}_{\ov x} \xi, q(n_k,n_k,\ov x,\xi )))} {\ov M_{i,\ov f^{LJ} \ov x} (B(F^{(LJ)}_{\ov x} \xi, q(J,n_k,\ov x,\xi ))} \; +\; \e  .
  \end{eqnarray*}
  We use the ergodic theorem  to write the first line and the above expression for $g_{i,q (j, n_k, \ov x, \xi ) }(\ov F^{jL} (\ov x,\xi)) $ for the second line.
  The last line follows by (\ref{balls2}) and a telescoping argument.
  
  To finish the estimate, on the one hand, 
  we  have   \( \ov M_{i,\ov f^{LJ} \ov x} (B(F^{(LJ)}_{\ov x} \xi, q(J,n_k,\ov x,\xi )) )\le 1.\)
   On the other hand, by Lebesgue density theorem (see section \ref{Besicovich}), there are  constants $C_\beta $ and $\g_\beta $  such that
 the  set $A_\beta $ of points in $\ov \EE$ such that $\ov M_{i, \ov x}( B(\xi, \g')) \ge C_\beta  (\g')^d $  for all $0< \g' \le \g_\beta  $ has $\ov M_i$-measure at least $1-\beta /2.$ Therefore, for $\ov M_i$-a.e.$(\ov x,\xi),$ there are infinitely many indices $k$ such that $e^{(\g^- -2\e) Ln_k (\ov x,\xi) } <  \g_\beta  $ and \( \ov F^{(Ln_k(\ov x,\xi) )} (\ov x,\xi) \in A_\beta .\) 
 It follows that 
 \[  \liminf _{k \to \infty } -\frac{1}{n_k} \log \ov M_{i,\ov f^{Ln_k} \ov x} (B(F^{(Ln_k)}_{\ov x} \xi, e^{(\g^- - 2\e)L n_k (\ov x,\xi) } ))  \le - dL ( \g ^- -2\e ) .\]
 Therefore, \( - \int \log g_i^N (\ov x,\xi) \, d\ov M_i (\ov x,\xi) \le - dL ( \g ^- -3\e ) .\)
\end{proof}
 
 Next, we find $\beta >0$, $J \in \N$ and, for $\ov M_i$-a.e.$(\ov x,\xi),$ a sequence $n_k(\ov x,\xi)$, with upper density at least $\beta$, of measurable  positive functions   $q_{\e} (j,n_k,\ov x,\xi) , 0\le j\le n_k ,$ on $\ov \EE$  satisfying (\ref{sizeofq3}), (\ref{sizeofq4}) and (\ref{balls2}).  We are  given $\e >0$, $L \in \N$ and  $i \in I$ such that \[ -\int \log \| (D_\xi F^{(L)}_{\ov x} )^{-1}\| \, d\ov M_i (\ov x,\xi) \ge  L (\g ^- -\e/2 ) .\] 
 
We  can consider the constant $K$ defined by 
 \[K:=  \max _{(\ov x, \eta)}  |\log \|(D_\eta F_{\ov x}^{(L)} )^{-1}  \||.\]  
We use  Pliss lemma (see corollary \ref{Pliss2}) to find,   for $\ov M_i$-a.e.$(\ov x,\xi),$ a sequence $n_k (\ov x,\xi) ,$ $ 1\le k, $  of upper density at least $L\e/ (2K +2L)) =:\beta $ such that, for any $0\le j \le n_k,$
\begin{equation}\label{pliss2} \sum _{\ell=j}^{n_k-1}- \log \| (D_{F_{\ov x}^{(\ell L)}\xi } F^{(L)}_{\ov f ^{\ell L} \ov x} )^{-1}\| \; \ge  (n_k-j)L(\g^- -  \frac{3\e}{2}).\end{equation}
 For $n_k$ as above, we   define the sequence $q(j,n_k,\ov x, \xi)$ as follows:
 \begin{itemize}
 \item $q( n_k,n_k,\ov x,\xi) = e^{(\g ^- -2\e)L n_k}$  and,
 \item for $0\le j < n_k,$ \[ q (j, n_k ,\ov x,\xi ) :=e^{(\g ^- -2\e) Ln_k+ \frac{\e}{2} L(n_k-j)}\Pi _{\ell=j}^{n_k-1}  \| (D_{\ov F^{(\ell L)}_{\ov x} \xi } F^{(L)}_{\ov f ^{\ell L} \ov x} )^{-1}\|  .\]
 \end{itemize}
  Property (\ref{sizeofq3}) is verified by definition and, by (\ref{pliss2}), we can write   \[ -\frac{1}{L}\log  q (j, n_k ,\ov x,\xi ) \ge -(\g ^- -2\e) n_k-  \frac{\e}{2} (n_k-j)+  (n_k-j)(\g^- -  \frac{3\e}{2})  =- j(\g ^- -2\e),\] which is property (\ref{sizeofq4}).
  
 We choose $J$ so that relation (\ref{balls2}) follows from Taylor theorem. Indeed, we have  
 \[   q(j+1,n_k, \ov x,\xi) =  e^{-L\e/2} \| (D_{\ov F^{(jL)}_{\ov x} \xi } F^{(L)}_{\ov f ^{jL} \ov x} )^{-1}\|^{-1} q(j,n_k, \ov x,\xi) .\]
 Since $\| \Hess_{\ov F^{(jL)}_{\ov x} \xi } F^{(L)}_{\ov f ^{jL} \ov x}(v,v)\| $ is uniformly bounded  for all  vectors $v$ of norm 1 in $ T_{\ov F^{(jL)}_{\ov x} \xi } \EE_{\ov f ^{jL} \ov x}, \)   all $\xi \in \EE _{\ov f ^{jL} \ov x} $ and all $\ov f ^{jL} \ov x$, there is $\de'$ such that, if $\xi, \eta$ are both in $\EE_{\ov f ^{jL} \ov x}$ and satisfy  \(d_{\EE_{ \ov f ^{jL} \ov x}} (\ov F^{(jL)}_{\ov x}\xi, \ov F^{(jL)}_{\ov x}\eta) \le \de',\) 
we have \[ \frac{ \| (D_{\ov F^{(jL)}_{\ov x} \eta} F^{(L)}_{\ov f ^{jL} \ov x} )^{-1}\|^{-1}}{ \| (D_{\ov F^{(jL)}_{\ov x} \xi } F^{(L)}_{\ov f ^{jL} \ov x} )^{-1}\|^{-1}} \; \ge \; e^{-L\e /2} .\]
and therefore the relation (\ref{balls2}),
\[(F^{(L)}_{\ov f^{Lj} \ov x} ) B(F^{(Lj)}_{\ov x} \xi, q(j,n_k, \ov x,\xi ))) 
\; \supset B(F^{(L(j+1))}_{\ov x} \xi, q(j+1,n_k,\ov x,\xi ))) , \] 
holds as soon as \(q(j,n_k, \ov x,\xi )  \le \de' .\) 
But \(q(j,n_k, \ov x,\xi )  \le e^{ j L(\g^- -2\e) } \)  by (\ref{sizeofq4}). We can choose $J $ so that \(q(j,n_k, \ov x,\xi )  \le \de' \) for $j >J.$

We have proven  that for all $\e, L= L(\e), i \in I(L, \e)$ and $N$,  $$ - \int \log g_i^N (\ov x,\xi) \, d\ov M _i(\ov x,\xi)\le dL(-\g ^-  + 3\e).$$  The proposition follows in  by letting $N \to \infty$, averaging in $i \in I$ (see (\ref{infinite})), then dividing by $L(\e)$ and finally letting $\e \to 0.$
  \end{proof}
  
  \begin{proof}({\it{Proof of theorem \ref{IP}}})
  Theorem \ref{IP} follows from propositions \ref{invariance2}, \ref{martingale} and \ref{basicinequality4}: assume $ \g^-\ge 0  .$ Then, by Proposition \ref{basicinequality4}, $\kappa = 0.$ By proposition \ref{invariance2}, this implies that,  for  \(  \ov m\)-a.e. \( \ov x \in \ov X, \) 
 \( (\ov F_{\ov x})_\ast \ov M_{\ov x} = \ov M_{\ov f \ov x} ,\)  i.e. for $m$-a.e.$x \in X,$  \( \ov M_{\pi x} =  ((\ov F_{\pi x})^{-1})_\ast \ov M_ {\ov {f} \pi  x} .\) 
Iterating and using  proposition \ref{martingale}, we find that \( M_x =  \ov M_{\pi x } .\) \end{proof}

\section{Remarks and complements to theorem \ref{IP}}\label{IPbis} 

\subsection {Non-ergodic case}  In the setting of the IP, if the probability  measure $M$ is not ergodic, the same conclusion holds if 
 \( \int \min \{ \g ^-, 0\} \, dM =  0 \).  Indeed, in that case, write  $M= \int M_t \, dn(t) $ for the  decomposition of $M$ into ergodic measures, $\g^-_t$ the corresponding exponents. Since  \( \int \min \{ \g ^-, 0\} \, dM =  0 \), we have $\g ^-_t \ge 0 $ for  $n$-a.e. $t$. By theorem \ref{IP}, the disintegrations of  $n$-a.e. $M_t$ associated to the projection on $X$ are $\B$ measurable. But since $m$ is ergodic, the disintegrations of $M$ are given by the integrals of the disintegrations of $M_t.$ The conclusion follows.

\subsection {Regularity} \label{regularity} In the setting of the IP, exponents make sense as soon as \( \log ^+ \|D_\xi F_x\|\)   and  \( \log^+  \|D_{\xi} (F_{x})^{-1}\| \) are $M$-integrable. In \cite{AV10}, indeed, there is no condition on   \( \| \Hess_\xi (F_x)\|\)  and proposition \ref{basicinequality4} is proven under the condition that  \( \|D_\xi F_x\|, \;  \|D_{\xi} (F_{x})^{-1}\| \) are uniformly bounded. For that, there is another ergodic theory  argument to the effect that relation (\ref{balls2}) can be arranged eventually on the trajectory of $\ov M$-a.e. $(\ov x, \xi)$ (see \cite[ lemma 3.10 and claim (16)]{AV10}, see also the appendix of \cite{BDZ22}).
 
Theorems \ref{F} and \ref{malicet} hold with  weaker conditions of regularity. The  proofs  are indeed simpler for propositions \ref{basicinequality1} and \ref{basicinequality2} because we have, a priori, $\|A^{-1}\| \ge 1 $ and  $\|(T^{-1})\|_\infty \ge 1 $ so that the analog to relation (\ref{sizeofq4}) is automatic, whereas here we  use Pliss lemma to obtain (\ref{pliss2}) and thus  (\ref{sizeofq4}).

Theorem \ref{TY} is  stated here with the hypothesis  that the diffeomorphism is $C^{1+\alpha}$. This hypothesis  is used here twice: to obtain a Markov partition  for the hyperbolic map $(Z,\psi)$ and for proposition \ref{basicinequality3}. The existence of Markov partitions for $C^1$  Smale systems follows from Hiraide (\cite{H85}) and Sun (\cite{S91}). Moreover, in our setting, the decomposition \( y \mapsto E^u_y \oplus E^c_y \oplus E_y^s \) is  {\it{dominated}}  and the partition $\zeta ^u$ is increasing. The Wang, Wang and Zhu version of proposition  \ref{basicinequality3} (see \cite{WWZ18}) applies and  therefore  theorem \ref{TY} is valid if $\vf$ is of class $C^1.$

There are few examples of applications of an IP without the hypothesis that the space $X$ is compact. See section \ref{distinct} for the Teichm\"uller flow.

\subsection {Perturbations}\label{perturbations} In the setting of the IP, assume that the space $X$ is a compact metric space, that $f$ is a homeomorphism of $X$ and that the bundle objects $\EE _x , F_x $ depend continuously on $x$. Consider a sequence \( \{M_k \}_{k \ge 0} \) of $F$-invariant probability  measures on $\EE$  with \( \pi _\ast M_k = m \) for all $k.$ Assume there is a generating decreasing sub-$\s$-algebra \(\B\), that the mappings \(x\mapsto \EE_x,  x \mapsto F_x \) are \(\B\)-measurable, that there is a $F$-invariant probability measure $M$ such that \( M_k \) weak* converge to $M$ and that \( \int \min \{ \g ^-, 0\} \, dM_k \to 0 \) as \( k \to  \infty .\) \\
\begin{claim} With these hypotheses, the disintegration mapping \( x \mapsto M_x \) is $\B$-measurable. \end{claim}
\begin{proof} We want to show that the  entropy $\kappa $ associated to $M$ vanishes. By proposition \ref{basicinequality4} the entropies   $\kappa_k  $ associated to $M_k$ converge to 0. Let $\CC _n , n\ge 1,$ be a family of finer and finer finite partitions of $\EE$ that generate the Borel $\s$-algebra. For each $n$ we can define  approximate entropies $\kappa_{k,n}$ by  
\[ \kappa _{k,n} :=  \int \left( \int  -\log \E\left[\frac {d((\ov F_{\ov x}))^{-1}_\ast \ov M^k_{\ov f \ov x}}{d\ov M^k_{ \ov x }} (\xi)\Big|\CC _n\right] \, d \ov M^k_{\ov x } \right) \, d\ov m_(\ov x ).\]
By a classical result of Dobrushin (\cite{dobrushin}), the entropies $\kappa _{k,n}, \kappa _{\infty ,n} $ are nondecreasing in $n$  and
\[ \kappa _k = \lim\limits_{n \to \infty} \kappa _{k,n}, \quad  \kappa  = \lim\limits_{n \to \infty} \kappa _{\infty , n},\] 
The remark follows by choosing the elements of the partitions $\CC _n$ with negligible boundaries for all $n$ so that $\kappa _{\infty, n} = \lim\limits _{k \to \infty } \kappa _{k, n}$ for all $n$.
\end{proof}

\subsection {Deformations}
Let \(\pi :\EE \to X, f, F= \{F_x, x\in X \}, m, M, \pi _\ast M = m \) be as in the setting of the Invariance Principle. Assume that there is a measurable transformation \( H:\EE \to \EE\)  of the form  \( H(x, \xi) = (x,H_x (\xi) ) \)  such  that  \( H_x^{-1} \) is uniformly H\"older  continuous:
\[ {\textrm {there exist}} \, C, \beta \; {\textrm{ such that, for}} \; x\in M,\;  \xi,\eta \in \EE_x, d_{\EE}(\xi, \eta) \le  C (d_{\EE }(H_x \xi, H_x \eta))^\beta .\]
We say that \( \wt F := H \circ F \circ H^{-1}\) is a {\bf{deformation}} of $F$ if there is a family \( \wt F_x, x\in M\) such that \( \wt F (x,\xi) = (fx, \wt F_x (\xi)).\)\\
\begin{proposition}\label{deformation} Assume there is a generating decreasing sub-$\s$-algebra \(\B\), that \(\wt F\) is a deformation of $F$  and  that the mappings \(x\mapsto \EE_x,  x \mapsto \wt F_x \) are \(\B\)-measurable. Assume that \( \int \min \{ \g ^-, 0\} \, dM =  0 \) where \( \g ^-\) is the  exponent associated to $(X,F,M)$. Then, the mapping \( x \mapsto \wt M_x  := (H_x)_\ast M_x\) is \(\B\)-measurable. \end{proposition}
 \begin{proof} The proof is similar to the proof of theorem \ref{IP}.
 Observe that the \(\wt M_x\) are the disintegrations of a \(\wt F\)-invariant probability measure \(\wt M\). Consider the entropy \( \wt \kappa \). associated to the \( (\wt F, \wt M)\)  by the formula in definition \ref{Fentropy}. The basic inequality becomes 
 \begin{claim}\label{basicinequality5}\[ \wt \kappa \; \le -  \frac{d}{\beta}  \int \min \{ \g ^-, 0\} \, d\wt M .\]\end{claim}
\begin{proof}
By the usual argument, we can reduce to the ergodic case. We then  follow the proof of proposition \ref{basicinequality4}.   For the same reason, we can assume that the exponent is non-positive. Then, we have to estimate, for a suitable $L$\footnote{We write the argument in the case when $\wt M$ is $F^{(L)}$ ergodic. The general case is treated  as in  the proof of proposition \ref{basicinequality4}}
\[ L \wt \kappa = -\int \log \wt g(\ov x, \eta) \, d\wt M_{\ov x} (\eta) d \ov m (\ov x),\] where \[ \wt g(\ov x, \eta ) := \frac {d((\wt F^{(L)}_{\ov x})^{-1})_\ast \wt M_{\ov f^{L}\ov x} }{d\wt M_{\ov x}}(\eta) = \frac {d((\ov F^{(L)}_{\ov x})^{-1})_\ast \ov M_{\ov f^{L}\ov x} }{d\ov M_{\ov x}}(H_{\ov x}^{-1} \eta).\]
For all $N$, we approximate the integrable function $\wt g^N = \sup \{ \wt g, e^{-N} \} $ and we find, in the same way as in the proof of proposition \ref{basicinequality4},  a number $\beta >0$ and  a suitable sequence $n_k, k\in \N,$ of  upper density at least $\beta $ such that 
\[ -\int \log \wt g^N (\ov x, \eta) \, d\wt M (\ov x,  \eta) \le   \liminf _{k \to \infty } -\frac{1}{n_k} \log \ov M_{\ov f^{Ln_k} \ov x} (B(\ov F^{(Ln_k)}_{\ov x} (H_{\ov x}^{-1}\eta), e^{(\g^- - 2\e)L n_k } )) + \e\]
Observe that \[ \ov M_{\ov f^{Ln_k} \ov x} (B(\ov F^{(Ln_k)}_{\ov x} (H_{\ov x}^{-1}\eta), e^{(\g^- - 2\e)L n_k  } )) =  \wt M_{\ov f^{Ln_k} \ov x} (B(H_{\ov f^{Ln_k} \ov x}^{-1}  \wt F^{(Ln_k)}_{\ov x} (\eta), e^{(\g^- - 2\e)L n_k  } ))\] and  that, by our hypothesis,
\[ B(H_{\ov f^{Ln_k} \ov x}^{-1}  \wt F^{(Ln_k)}_{\ov x} (\eta), e^{(\g^- - 2\e)L n_k  } ) \supset B(\wt F^{(Ln_k)}_{\ov x} (\eta), C^{-1}e^{\frac{1}{\beta}(\g^- - 2\e)L n_k  } ).\]
Relation  (\ref{basicinequality5}) follows since we can find infinitely many $n_k$ such that 
\[ -\log \wt M_{\ov f^{Ln_k} \ov x} (B(\wt F^{(Ln_k)}_{\ov x} (\eta), C^{-1}e^{\frac{1}{\beta}(\g^- - 2\e)L n_k  } )) \le -\frac{d}{\beta}(\g^- - 3\e)L n_k .\]
 \end{proof}
 Proposition  \ref{deformation} follows. \end{proof}

\subsection  {Comments}\label{comments}

Theorem \ref{F} flollows from  theorem \ref{IP}: take  \( X= (\SL(2,\R))^{\otimes \Z},\) \( m= \mu ^{\otimes Z}, \EE = X \x \PP ^1(\R).\) Then, consider \( F(x,\xi ) = (fx, x_0 \xi), \; F^{-1} (x,\xi ) = ( f^{-1}x, (x_{-1})^{-1} \xi).\) If the exponent of the product of matrices is 0, for any ergodic $F$-invariant probability measure $M$ on $\EE,$ \( \g ^+ = \g^- =0.\) Applying theorem \ref{IP}  to  $F$ for the  $\s$-algebra $\B$ generated by the nonnegative coordinates of $X$ and to $F^{-1} $ for  the  $\s$-algebra $\B'$ generated by the negative coordinates of $X,$ we conclude that \( x \mapsto M_x \) is measurable with respect to both \(\B\) and \(\B'\). By independence of  \(\B\) and \(\B'\),  there is \( \nu \) on \(\PP^1\) such that \(M_x = \nu \) for $m$-a.e. $x$. By (\ref{invariance}), we have \( x_0 \nu = \nu \) for \(\mu\)-a.e. $x_0$.

Theorem \ref{IP} extends theorem \ref{F} in the different directions seen in section \ref{examples}.  The form of dependence in theorem \ref{IP} for linear cocycles was formulated in \cite{L86}, in order to get a common reason for theorem \ref{F} and  theorem \ref{kotani} (see \cite[ section IV]{L86}). See section \ref{random} for the applications of theorem \ref{IP} to compositions of independent  actions in real  dimensions higher than 1. Theorem \ref{TY} is  not  a formal  consequence of theorem \ref{IP}, see \cite{AVW2} for an application of  theorem \ref{IP} to compact extensions of hyperbolic systems.

 Another extension  of theorem \ref{F} has been initiated by  I.~Ya.Gol'dsheid (\cite{G22}), where he considers independent (and Markov) products of non identically distributed matrices. In particular, A. Gorodetski and V. Kleptsyn obtain a version of theorem \ref{F} and several  consequences, for example for non-stationary Schr\"odinger equations (see \cite{GK22}, \cite{GK25}). It is mentioned here 
 because their proof can be interpreted as a one-step IP and also because their work  is related to non-autonomous dynamics.
 
\section{Some examples}\label{examples2}
\subsection{Hyperbolic systems}
\subsubsection{Invariant foliations}

Recall the  definition  \ref{PH} of a partially hyperbolic $C^2$-diffeomorphism of a compact  manifold $X$. Applying the proposition \ref{stablemanif} to $f$ and $f^{-1}$, we obtain a partial stable (respectively unstable)  lamination $\W^\rho $ (resp. $\W^{\rho '}$ with a suitable $\rho'$), with local holonomies $h^\rho_{x,y} , y \in W^\rho _{loc} (x) $  (resp. $h^{\rho'}_{x,y} , y \in W^{\rho '} _{loc} (x) $).  Recall also that we call the system hyperbolic if it is partially hyperbolic with $\dim E^c =0$. Then, there is $\rho$ (resp.$\rho '$) such that 
 $\W^\rho $ (resp. $\W^{\rho '}$) is the stable $\W^s $ (resp. unstable $\W^u$) foliation. 
 
If the system $(X, f)$ is hyperbolic and $Q$ a Markov partition, then $\B := \vee _{n=0}^\infty f^{-n}Q $  is an increasing generating $\s$-algebra to which the hypotheses of theorem \ref{IP} may apply. In the case of partially hyperbolic systems, it is possible to construct  measurable partitions of $X$ with a {\it{premarkov property}} so  that the associated $\s$-algebra can play the role of $\B.$ Namely:
\begin{proposition}\label{premarkov} Let $(X,f) $ be a partially hyperbolic $C^2$-diffeomorphism. Given  $\rho <1,$ let $\W^\rho $ be the lamination obtained in proposition \ref{stablemanif}. 
Let $m$ be a $f$-invariant ergodic probability measure
and $\e >0$ small enough.  There exists a measurable partition $Q$ of $X$ such that 
\begin{enumerate}
\item for $m$-a.e.  $x\in X,$ the element $Q(x) $ containing $x$ is an open set of $W^\rho _{x, loc}  $ of  inner diameter at least $\e$, 
\item for $m$-a.e. $x \in X \; f(Q(x))  \subset Q(f x).$
\end{enumerate}
\end{proposition}

Let $Q$ be a  partition given by proposition \ref{premarkov}, $x \mapsto m^Q_x $ a family of conditional measures for $Q$. Then,  for $m$-a.e. $x$, $Q(x) $ is $m^Q_x$-essentially partitioned by the images $f Q(y), y \in X$ that contain some open subset of $Q(x).$ The $\s$-algebra $\B$ associated to $Q$ is thus non-increasing. Moreover, if $ f^{n} y \in Q(f^{n} x) $ for all  $n \in Z,$ then $d_{W^\rho} (x,y) = 0 $ and thus the $\s$-algebra $\B$ is generating for $f^.$ We can use this $\s$-algebra $\B$ in theorem \ref{IP}.

A more precise premarkov partition is constructed in \cite[ propositions 4.4 and 4.5]{AV10}, where similar  properties are  obtained for all $x \in X$ and not only for $m$-almost every $x$. 
Proposition \ref{premarkov} with $m$ the Lebesgue measure  goes back to Sinai (\cite{sinai}). See proposition \ref{LS}  for the general case of Pesin invariant manifolds and any invariant probability measure.

 For hyperbolic systems,  $\W^s $ and $\W^u$ form locally a system of coordinates. More precisely, 
 \begin{proposition}\label{Bowenproduct}(\cite{B75}) Let $(X,f) $ be a transitive  hyperbolic $C^2$-diffeomorphism. There is $\de >0 $ such that, if $d(x,y) < \de,$ then  $\W_{x,loc}^s  $ and $\W_{y, loc}^u$  intersect at exactly one point, and this point $[x,y]$ depends continuously on $x,y.$ \end{proposition}

For hyperbolic systems, the notion of $(s,u)$-product probability measure is too strong, since  there is only one 
$(s,u)$-product probability measure  (see definition \ref{suproduct} and the discussion there). We define the notion of {\it{ $(s,u)$-product structure.}}

\begin{definition}\label{almostsuproduct} Assume that $(X,f) $ is  hyperbolic. A measure $\nu$ is said to have a {\bf{$(s,u) $-product  structure}} if, in a common local chart for both foliations,
\begin{itemize} \item the local holonomies $h^s$ preserve the sets of measure 0 for the conditional $\nu ^u_y$ of the measure $\nu $ with respect to the partition in local  \( W^u \)-leaves  or, equivalently,
 \item the local holonomies $h^u$ preserve the sets of measure 0 for the conditional $\nu ^s_y $ of the measure $\nu $ with respect to the partition in local  \( W^s \)-leaves.
 \end{itemize}\end{definition}
 
 The equivalence of the two statements is a direct application of Fubini Theorem: in a local chart with two laminations, the measure $\nu$ has the same negligible sets as the product of both transversal  projections of any fiber measure. The measure of maximal entropy is a $(s,u)$-product measure and so, of course, has a  $(s,u)$-product structure, but there are many others: for any H\"older continuous function $F$ on $X$, the measure that realizes the maximum of $\mu \mapsto h_\mu (f) + \int F d\mu $ over invariant probability measures $\mu$, where $h_\mu (f)$ is the entropy of the transformation $f$ for the measure $\mu,$  has a $(s,u)$-product structure (see \cite{B75}). Any invariant probability measure absolutely continuous with respect to the Lebesgue measure also has a  $(s,u)$-product structure (\cite{A67}). More generally, an invariant measure is called SRB (for Ya. G. Sinai, D. Ruelle and R. Bowen) if the  conditionals on local unstable manifolds are absolutely continuous with respect to the Riemannian volume. Since the stable holonomies preserve the Riemannian volume class, a hyperbolic SRB measure has a   $(s,u)$-product structure as well.

\subsubsection{Holonomy invariance}
Let \((X, \A, f, m, \EE, F)\) be as in the setting of the IP with $f$ a $C^2$ partially hyperbolic diffeomorphism of the compact manifold $X$ and $F$ a $C^{1+\alpha } $ cocycle. Take $\rho <0$ with nontrivial $\W^\rho $ and  assume that the holonomies $h^\rho_{x,y} $ are defined for all $x \in X,$ all $  y \in W^\rho _x.$ We say that $F$ is {\bf{compatible}} with the holonomies \( h^\rho \) if, for  all $x \in X,$  all \( y\in W^\rho _x,\) there exists an application  \( \wt  h^\rho_{x,y}: \EE_x \to \EE_{y} \) such that 
\begin{itemize}
\item \( F_{y} \circ  \wt  h^\rho_{x,y} \; =\;  \wt  h^\rho_{f(x),f(y)} \circ F_x, \) 
\item there is $C, \beta\) such that for  \( d_{\EE_x}(\xi, \eta ) \le  \delta,\) \( d_{\EE _y}( \wt  h^\rho_{x,y}\xi, \wt  h^\rho_{x,y}\eta ) \le C (d_{\EE_x}(\xi, \eta ))^\beta,\) 
\item \((x,y, \xi) \mapsto  \wt h^\rho_{x,y} (\xi )\) is continuous and
\item we have  \(  \wt  h^\rho_{y,z} \circ  \wt  h^\rho_{x,y} \;= \;  \wt  h^\rho_{x,z} ,\quad  \wt  h^\rho_{x,x} \; =\; \mathbb {Id}. \)
\end{itemize}

These conditions ensure that there is a $F$-invariant  H\"older continuous foliation \( \wt \W^\rho \) of \(\EE\) and that the  leaves \( \wt W^\rho (x, \xi) \)  homeomorphically project  to \( W^\rho (x),\) for all \((x,\xi) \in \EE.\)
We have the following topological version of proposition \ref{deformation}:
\begin{proposition}\label{holonomy}(\cite[corollary 4.3]{AV10})
With the above notations, assume that $F$ is compatible with the holonomies \( h^\rho . \) Let $M$ be a $F$-invariant measure with disintegrations \( M_x, x\in X.\)  If \( \int \min \{ \g ^-, 0\} \, dM =  0,  \) then there is a \( \W^\rho \)-saturated subset $E$ of $X$ with full $m$-measure and a version of the disintegration $M_x$ 
such that, if \( x,y \in E\) and are in the same $\W^\rho $-leaf, \begin{equation}\label{stableinvariance} (\wt h^\rho _{x,y} )_\ast  (M_x)\; =\; M_y.\end{equation}
\end{proposition}

\begin{proof}

{\it{Step 1.}} Let $P$ be a measurable partition with the  premarkov property of proposition \ref{premarkov} and let $\B$ be the  associated non-decreasing generating $\s$-algebra.
Choose a $\B$-measurable section \( p:M \to M \) such that for $m$-a.e. \( x \in X, \; p(x) \in \B (x).\) Using the holonomies \(\wt h ^\rho_{x,p(x)},\) we can identify the fibers so that  \(\EE _{p(x)} = \EE _x.\)

{\it {Step 2.}} Define, for \(x \in X, \;   H : \EE \to \EE \) by \( H(x,\xi) = (x,H_x(\xi )) := (x,\wt h^\rho_{x, p(x) } \xi) .\) Then, \( G := H^{-1} \circ F \circ H\) is a  deformation of \(F\) such that \( x\mapsto G_x \)  is $\B$-measurable. Then, by proposition \ref{deformation}, the mapping \( x \mapsto \wt M_x  = (H_x)_\ast M_x\) is \(\B\)-measurable.

{\it {Step 3.}} It follows that  there is a set of full measure $E_0$ such that if $y,z$ are in $E_0$ and in the same atom of $\B$, $p(y)=p(z) $ by definition, and $\wt M_y = \wt M_z $ by Step 2. Thus, for such $y,z,$
\[ M_y = (\wt h^\rho_{p(y),y})_\ast \wt M_y = (\wt h^\rho_{p(z),y})_\ast \wt M_z =(\wt h^\rho_{z,y})_\ast  (\wt h^\rho_{p(z),z})_\ast \wt M_z = (\wt h^\rho_{z,y})_\ast  M_z.\]

{\it {Step 4.}} Let \( E_1 := \cap _{n \ge 0 } f^{-n} E_0.\)  The set $E_1$ has full $m$-measure  and, if $y,z$ are in $E_1$ and in the same  \( \W^\rho \)-leaf, there is an $n \ge 0 $ such that  $f^ny,f^n z$ are in $E_0$ and in the same atom of $\B$. By Step 3, \( (\wt h^\rho _{f^n x,f^n y} )_\ast M_{f^n x} = M_{f^n y}.\) (\ref{stableinvariance}) follows on $E_1$ by (\ref{invariance}) and the equivariance of the \( \wt h^\rho .\)

{\it{Step 5.}} Set \( M_x := (\wt h ^\rho _{y,x})_\ast M_y.\) for any \( y \in W^\rho(x) \cap E_1  . \)  By Step 4, the family \(x\mapsto M_x\) is defined $m$-a.e. and is a family of disintegrations of $M$ for $\pi$ that satisfies (\ref{stableinvariance}).
\end{proof}

\subsubsection{Applications}\label{positivity}
Let \((X, \A, f, m, \EE, F)\) be as in the setting of the IP with $f$ a $C^2$ hyperbolic diffeomorphism of the compact manifold $X$ and $F$ a $C^{1+\alpha } $ cocycle.  Let $\W^s, \W^u $ be the stable and unstable foliations of $(X,f)$. They satisfy the product property of proposition \ref{Bowenproduct}. Assume also that $F$ is compatible  with the holonomies $h^s$ and $h^u.$

Let $M$ be an ergodic $F$-invariant measure  such that $\Pi_\ast M=m.$ Assume that both exponents $\g ^+$ and $\g^-$ are 0. Then, by proposition \ref{holonomy}, there are sets $Z^u \subset X$ and $Z^s \subset X$ of full measure such that the disintegrations $M_x, x \in X$ are 
 $m$-almost everywhere $\wt h^u$-equivariant  along  unstable leaves on $Z^u$ and $\wt h^s$-equivariant along stable manifolds on $Z^s.$ Assume moreover that $m$ has a  $(s,u)$-product structure (see definition \ref{almostsuproduct}). Then, using proposition \ref{Bowenproduct}, one sees that  there is a version of the disintegration $ M_x $ that is continuous in $x$ and holonomy-equivariant. This is the starting point for many applications. For example, we can now discuss the generalization of  Furstenberg theorem \ref{F} to general H\"older continuous cocycles over a hyperbolic system:
 \begin{theorem}[\cite{V08}] \label{holdercocycle} Let $(Y,f)$ be a topologically mixing $C^\infty $ hyperbolic system and $A: Y \to \SL(2,\R) $ a  $C^{1+\alpha} $  cocycle. Let $m $ be an $f$-invariant measure with a  $(s,u)$-product structure and $\chi \ge 0$ the $m$-a.e. value of \[ \lim\limits_{n \to \infty }\frac{1}{n} \log \| A(f^n y) \circ \ldots \circ A (y)\| .\]
 If $\chi = 0,$ 
 there exists a continuous $B : Y \to \SL (2, \R) \) such that \( (B (y))\circ A(y) \circ ( B( fy))^{-1} \) is either the identity  for all $y$, or a nilpotent matrix  for all $y$ or a rotation for all $y$. 
 \end{theorem}
 \begin{proof}We indicate the main steps of the proof, referring to \cite{V08} for details.   We consider the system $(Y,f, m),$ the bundle  $\EE := Y \times \PP^1$ and the H\"older continuous bundle of diffeomorphisms given by $\F_y ([\xi]) := [A(y) \xi ].\) 
 
 We have $ |D_\xi F_y | \; = \; \frac {\|\xi \|^2}{\|A\xi \|^2} $ and thus, for any $\F $-invariant measure $M$ that projects on $m$, the exponent $\g ^- = -2\chi = 0.$  We may assume $M$ ergodic. Let $M_y$ be a family of disintegrations of $M$ for $\Pi : \EE \to Y.$  By (\ref{invariance}), $A(y) _\ast M_y = M_{fy} $ and therefore the type of $M_y $ is $f$-invariant.  By ergodicity, this type is $m$-a.e. constant and either  a finite sum of equal Dirac measures or continuous.

 In order to apply proposition \ref{holonomy}, we have to find holonomies $\wt h^s, \wt h ^u$ compatible with $\F$. In other words, we have to lift the foliations $\W^s, \W^u $ to H\"older continuous stable and unstable foliations for $\F$ on $\EE .$ We cannot apply proposition \ref{stablemanif} because the system $(\EE,\F )$ is not partially hyperbolic: there is a  continuous decomposition with $\wt E^s ,    \wt    E^c =  \{0\} \otimes T\PP ^1, \wt E^u $ with uniform contraction/expansion  in $\wt \E^{s,u},$ but the expansion/contraction in the central direction is not uniform, and not, a priori, dominated by the other  contraction/expansion coefficients. Such a property holds if one assumes a {\it{bunching}} condition on the action of $DF$ on the bundle $E^c$. But in any case, the exponent in the central direction is 0 and we can use Pesin theory to find  $M$-a.e. stable and unstable foliations for $\F$ on $\EE .$ In terms of local holonomies, we find holonomies $\wt h^s, \wt h ^u$ compatible with $\F$, but 
  the H\"older continuity property becomes, for $m$-a.e. $y \in Y,$ and all $z \in W^\rho _{y, loc},$ 
there is $C(y), \beta$  such that for $  \xi, \eta  \in \PP ^1,$ we have \[ \; d(\xi, \eta ) \le  \delta \Longrightarrow d( \wt  h^\rho_{y,z}\xi, \wt  h^\rho_{y,z}\eta ) \le C(y) (d(\xi, \eta ))^\beta .\]  
 Moreover, the function  $C$ is positive $m$-a.e. and satisfies $ C(f^\pm y) \le e^\e C(y) .$
 The continuity of the mappings $(x,y) \mapsto \wt h^s_{x,y}, \wt h ^u_{x,y} $ are also localized: there are compact sets $K_\ell \subset Y$ such that $\lim\limits _{\ell \to \infty } m(K_\ell) =1,$ $C$ is positive on $K_\ell $ and $(x,y) \mapsto \wt h^s_{x,y}, \wt h ^u_{x,y} $ are continuous on $K_\ell.$
 The arguments of the proofs of propositions \ref{deformation} and \ref{holonomy} still work and we obtain that there are versions of the conditional measures $M_y$ on $\PP ^1$ that are holonomy equivariant  $m$-a.e. on stable manifolds and versions of the conditional measures $M_y$ on $\PP ^1$ that are holonomy equivariant  $m$-a.e. on unstable manifolds. Using the $(s,u)$-product structure of the measure $m$, we see that there is a version of $y \mapsto M_y$ that is continuous  and holonomy equivariant on $[K_\ell , K_\ell ]$ and therefore  constant on $\cup _\ell [K_\ell, K_\ell]$. Let $\nu $ be the  measure on $\PP ^1$ such that for $y \in \cup _\ell [K_\ell, K_\ell]$, $M_y = \nu.$ We have, for $y \in \cup _\ell [K_\ell, K_\ell],$ \begin{equation}\label{invariance4} A(y) _\ast \nu \; = \; \nu.\end{equation}
  \cite{V08} concludes by exploiting equation(\ref{invariance4}) at periodic points in $[K_\ell, K_\ell]$ and using e.g. \cite{AKL18}. \end{proof}
There are many refinements and variants
 (see e.g. \cite{BGMV}, \cite{V08}, \cite{AV10}, \cite{ASV}, \cite{AVW2} , \cite{CMY} and \cite{ADZ23}). Often, in a space of  highly  regular linear cocycles, one finds that the set of   cocycles
 in $\SL (2,\R)$, say, with positive exponent is generic in some open set, a sharp contrast with the genericity of 0 exponent in low regularity (see \cite{bochi}, \cite{bochiviana}).

 On the other hand, with a different approach, U. Bader and A. Furman (see \cite{BF25}) can prove a more general form of  theorem \ref{holdercocycle}, see our comment at the end of section \ref{distinct}. It might be possible to use \cite{BF25}'s formalism in some other  applications.

\subsection{Random diffeomorphisms}\label{random}
In this section, we consider a (compact) manifold  $Y$, a finite Riemannian volume $\la$ on $Y$  and a probability measure \( \mu \) on the group \( \DD ^2_\la (Y)\) of $C^2$-diffeomorphisms of $Y$ preserving $\la .$  We form the random product \[ g^{(n)} (y) \;:= \; g_{n-1} \circ \ldots \circ g_0(y) ,\]
where the \( g_i \in \DD^2_\la (Y) \) are i.i.d. with distribution \(\mu.\)
We say  that $\mu$ has finite first moments if we have \[ \int _Y \log \| D_y g\| \, d\mu<+ \infty \; {\textrm{ and }}  \int _Y \log \| (D_y g^{-1})^{-1} \| \, d\mu <+ \infty. \] The exponents \( \g ^\pm (y) \) are given, for \( (\mu ^{\otimes \Z } \times \la) \)-a.e.
 \( (x,y )  \in (\DD ^2_\la (Y))^{\otimes \Z } \times Y,\) by 
 \[ \g ^+(y) \,:=\, \lim\limits _{n \to \infty } \frac{1}{n} \log \| D_y g^{(n)} \|,  \quad  \g ^-(y) :=\, -\lim\limits _{n \to \infty } \frac{1}{n} \log \| (D_y g^{(n)} )^{-1} \| .\]
(It is easy to see that the limits exist \( (\mu ^{\otimes \Z } \times \la) \)-a.e.
 and are \( \mu ^{\otimes \Z } \)-a.e. constant). Observe that  \( \g^+(y) \ge \g^- (y)\) for \(\la\)-a.e. \(y\) and that, since the diffeomorphisms preserve $\la ,$ \( \int _Y \g ^- (y)\, d\la(y) \, \le \,0.\)

We assume that the action of $\mu$ is ergodic: there is no non-trivial measurable subset of $Y$ which is invariant under \(\mu\)-a.e. $g.$ Then, the exponents \( \g ^\pm (y) \) are  constant  $\g^\pm $ \(\la\)-a.e.. If $W$ is  a measurable linear subbundle $W \subset TY, W \not = \{0\}$ which is invariant by the action of $Dg$ for $\mu$-a.e. $g$, we can define in the same way the exponents $\g ^\pm _W$ restricted to $W.$

\begin{definition} Let $\mu$ be a probability measure on \( \DD _\la^1 (Y)\), with finite first moment.  Set \( \mu ^{(n)} \) for the distribution of \( g^{(n)}.\) 
We say that \(\mu\) is {\it{uniformly expanding}} if there exist  \( n_0\in \N_{\ge 0} ,  C>0 \) such that, for {\it { all }} \( y\in Y,\) {\it { all }} \( v\in T_yY, \) \[ \int _{\DD^1_\la (Y)}  \log \| D_y g (v) \| \, d\mu ^{(n_0)} (g)  > C.\] \end{definition}

{\bf{Fact}} (\cite{C06}, see \cite{CD25} for  a proof) {\it {A measure \(\mu\) on \( \DD ^1_\la \) with first moment and acting ergodically on $Y$ is uniformly expanding if, and only if, there exists no measurable subfield $W \subset TY, W \not = \{0\},$ which is invariant by the action of $Dg$ for $\mu$-a.e. $g$ and such that the exponent  $\g^+_W$ vanishes.}}

Uniform expansion is the key property to establish {\it{stiffness}} of the action of \(\mu\), namely that all stationary measures are invariant. This has been established by A. Brown and F. Rodriguez Hertz (\cite{BRH17}) when $Y$ is a real surface, S. Cantat and R. Dujardin (\cite{CD25})  for automorphisms groups of projective complex surfaces,  A. Brown, A. Eskin, S. Filip and F. Rodriguez Hertz (\cite{BEFRH}) for higher dimension real manifolds (in this case, one has to consider uniform expansion on all exterior products \( \bigwedge ^k TY\), see \cite{ES23}). The proofs use the factorization method of  A. Eskin and M. Mirzakhani (\cite{EM18}).

The IP is used to understand  the case when \(\mu\) is not uniformly expanding. We form \( X:=  (\DD ^2_\la (Y))^{\otimes \Z } \times Y,\; m= (\mu ^{\otimes \Z } \times \la),\) and \(f (\om,y)= (\s \om, g _0 y),\) where $\s$ is the shift transformation on \( (\DD ^2_\la (Y))^{\otimes \Z }.\) We form the  bundle \( \pi : \EE \to X,\) where \(\EE := \left((\DD ^2_\la (Y))^{\otimes \Z } \times \PP(W)\right), \) and  the  bundle $C^1$-diffeomorphism $F$ defined, for $\om =\{g_n\}_{n \ge 0} \in (\DD ^2_\la (Y))^{\otimes \Z } $ and  $[ v] \in \PP(W_y)\)  by 
\( F (\om, y, [v]) := [{D_y g_0 (v)}].\)

 \begin{proposition}\label{fiberinvariance}  Assume that $\mu $ has bounded support \footnote{We apply theorem \ref{IP} to the derivative action on $\PP (W)$. Given the form of the statement of theorem \ref{IP} we gave, we have to assume bounded support  on  \( \DD ^3_\la (Y)\) (see the discussion in  section \ref{regularity}).}  on  \( \DD ^3_\la (Y)\)  and that  \( \g^\pm_W = 0 \; \la\)-a.e.. Then any $F$- invariant probability  measure that projects on $m$ is of the form  \( M = (\mu ^{\otimes \Z } \times \La )\),  where   \(\La \) is a measure on \(\PP (W)\) with disintegration \( \La = \int _Y \La _y \, d\la (y) \) in such a way that, for \(\mu\)-a.e. \(g,\)  \begin{equation}\label{invariance3} (D_yg)_\ast \La _y \;=\; \La _{gy} .\end{equation} \end{proposition}

{\it {Proof:}} We may assume  $M$ ergodic. Observe that the exponents \( \G^\pm (\EE,F, m) =0 \) and thus  \( \G^- (\EE,F^{-1}, m)  = - \G^+ (\EE ,F, m)  = 0 .\)
 We may apply the Invariance Principle  on the one hand, to $F$ and the $\s$-algebra $\B $ generated by \( \{g_n, n\ge 0\}\)  and \( y,\) and on the  other hand  to $F^{-1}$ and the $\s$-algebra $\B '$ generated by \( \{g_n, n<  0\}\) and \(y.\) We obtain that the disintegrations \(M_{\om,y}\) depend only on \(y\). We set, for \(\la \)-a.e. \(y,\) \( \La_y := M_y.\)

 We have \( M =(\mu ^{\otimes \Z } \times \int _Y  \La_y \, d\la (y)  ).\) The relation (\ref{invariance}) 
yields that, for \(\mu \times \la \)-a.e. \((g, y),\)  \( (D_y g)_\ast \La _y \;=\; \La _{gy} . \qed \)

Remains to understand the property (\ref{invariance3}). This might be simplified by the cocycle reduction theorem.
\begin{theorem} [\cite{ACO97}]  Let \( \mu\) be a probability measure on \( \DD^2_\la (Y), W \subset TY\) a $\mu$-a.e. $Dg$-invariant field of subspaces. Then, there exists a measurable mapping \( Q : (Y \times \PP^{\dim W -1}) \to \PP W, \) with \( Q  (y, \xi ) = (y, Q_y (\xi)) \) and \( Q_y \) is a projective mapping from  \(\PP^{\dim W-1}\) to  \(\PP W_y\) such that any family \( \nu _y, y\in Y\) of measures on \(\PP^{\dim W-1}\) that satisfies \[ ([D_y g] \circ Q_y)_\ast \nu _y = (Q_{gy} )_\ast \nu_{gy} \quad {\textrm {for }}  (\mu \times \la ){\textrm{-a.e. }} (g,y) \] is \(\la\)-a.e. constant. \end{theorem}

It follows that, if  (\ref{invariance3}) is satisfied, there is a  measure \( \nu \) on \( \PP^{\dim W-1} \) such that \( \nu _y = (Q_y)_\ast \nu \) for \(\la\)-a.e. $y$ and the projective mappings \( Q_{gy}^{-1} \circ [D_y g] \circ Q_y \) (\(\la\)-almost) all belong to the group of projective mappings  preserving \(\nu.\) By studying the possible subgroups of projective transformations preserving a probability  measure on \(\PP^{\dim W-1},\) one can dynamically characterize the non-uniformly expanding case (\cite{BRH17}, \cite{CD25}, \cite{BEFRH}).  One can also prove  (\cite{BM20}), see also  \cite{OP22}) that a generic (in some sense) measure \(\mu \) on \( \DD^2_\la (Y) \)  is uniformly expanding and gives rise to  positive exponents.
 This extends \cite{BGMV}, \cite{V08} and many other partial results.

\section{Dynamically coherent partially hyperbolic systems}\label{examples3}

Recall the  definition \ref{DC} of a dynamically coherent partially hyperbolic dynamical system. Recall that the H\"older continuous  lamination \(\ov \W^c \) with leaves the connected components of 
 \[ \ov W^c_x \; :=\;  W^{cs}_x \cap W^{cu}_x \]  integrates the direction $E^c_x$ and that we denote  $h^c$  the holonomy mappings between local  transversals to $\ov \W^c.$ 
 The mappings $h^c $ preserve the local $
 \W^{cs}, \W^{cu} $ leaves. For the sake of presentation, we assume in this section 
 that the holonomies  $h^s$ and $h^c$  (respectively $h^u$ and $h^c$) commute.  For such systems, we have the notion of partial $(s,c,u)$-product (definition \ref{scuproduct}, see theorem \ref{TY} for examples). We now  define the  weaker notion of   partial $(s,c,u)$-product structure.  \begin{definition}\label{scuproductstructure} Assume that $(Y,\vf) $ is partially hyperbolic and dynamically coherent. A measure $\nu$ is said to have a  {\bf{partial $(s,c,u) $-product  structure}} if
\begin{itemize} \item the conditionals $\nu ^{cs}_y $ of the measure $\nu $ with respect to the partition in local  \( W^{cs} \)-leaves have a  $(s,c)$-product structure  and
 \item the conditionals $\nu ^{cu}_y $ of the measure $\nu $ with respect to the partition in local  \( W^{cu} \)-leaves have a  $ (c,u)$-product structure.
 \end{itemize}\end{definition}
 Let $\ov \W^{su}$ be  a local foliation, locally transverse to $\ov \W^c.$  We define  $\ov \W^s$ and $\ov \W^u$, transverse foliations of each $\ov W^{su} $-leaf, as the traces of $\W^{cs} $ and $\W^{cu} $. The foliations  $\ov \W^s, \ov \W^u$ and $\ov \W^c $ form  a local system of coordinates.
If a measure has a   partial $(s,c,u)$-product structure, then the projected measures on the leaves of  $\ov \W^s$ and $\ov \W^u$ are related (by the local center holonomies) to the conditional measures on the original $ \W^s$ and $\W^u$.

 \begin{definition}  We say that the probability $\nu$ has   a   {\bf{$(s,c,u)$-product  structure}} if, locally, given a local foliation $\ov \W^{su}$ as above, all holonomies $h^\ast $, where $\ast = s,u,c, cs, cu,$ preserve the 0-sets of the conditional measures on the local leaves of $\ov \W^{\ast '}$, where $\ast' =cu, cs, us, u, s$ respectively.\footnote{With the convention that $\ov \W^{cu}:=\W^{cu}, \ov \W^{cs} := \W^{cs}.$} \end{definition}
  In other words, a measure $\nu$  has a   $(s,c,u)$-product structure  if it is locally equivalent to a product in the coordinates given by $\ov \W^c, \ov \W^s$ and $\ov \W^u$.
 We have
\begin{proposition}\label{product}  If the measure $ \nu$  has  a partial $(s,c,u)$-product structure and, locally,  the transverse measures to $\W^c$ has a $(s,u)$-product structure with respect to $\ov \W^s$ and $\ov \W^u$, then the measure $\nu $ has  a  $(s,c,u)$-product structure. \end{proposition}
This follows from  lemma \ref{quasiindependence}.
Using the  absolute continuity of the foliations $ \W^s$ and $ \W^u$  (\cite{A67}),  lemma \ref{quasiindependence} also gives that   an  invariant probability measure absolutely continuous with respect to the Lebesgue measure has a   $(s,c,u)$-product structure as soon as the  conditionals on $\W^{cs} $ and $\W^{cu}$  themselves have a product  structure with respect to $h^s$ and $h^c$ (respectively $h^u $ and $h^c$).

  Examples of dynamically coherent partially hyperbolic systems may have  compact center leaves, have one-dimensional center leaves which are invariant (defined as {\it{ discretized geodesic flows}}), or have the property that the action on the central leaves is quasi-isometric (for instance, this is the case for non hyperbolic ergodic automorphisms of the torus). IP in these cases says that an invariant ergodic probability measure  such that the central exponents vanish has strong invariance property with respect to the center  foliation. 

  In the case of dynamically coherent partially hyperbolic systems with  compact center leaves,  the partition defined by the foliation $\ov \W^c $ is measurable and the associated  conditional measures  
 $\MMM _y^c $ are defined $m$-a.e. (\cite[Proposition 3.7]{AVW2}).
 In the general case of non-compact center leaves, the partition defined by the foliation $\ov \W^c $ may not be  measurable.
  We will still denote $\MMM_y^c $ the  family of local measures on the center  leaves, defined up to multiplication by a constant,  that are proportional to the local disintegrations (see section \ref{Rokhlin} and \cite[section 3.2]{AVW}).
 
   Theorem \ref{TY} shows that for compact extensions of  hyperbolic systems, ergodic invariant measures with 0 center Lyapunov exponents are  partial $(s,c,u)$-products. In other words, that the classes of measures $\MMM_y^c$ up to a constant factor are $h^u$-invariant for $m^{cu}$-a.e.$y$ and  are $h^s$-invariant for $m^{cs}$-a.e.$y$.
In this section, we survey some recent examples of applications of the IP to dynamically coherent partially hyperbolic  systems with some further features.

\subsection{Discretized Anosov  flows}

Let $g_t, t\in \R,$ be a one parameter group of smooth diffeomorphisms of a compact manifold $Y$.  It is called an Anosov flow if each $g_t$ is partially hyperbolic with the central direction being the one-dimensional tangent to the orbits of the flow. 

\begin{definition} Let $Y$ be a compact manifold, $g$ a smooth  diffeomorphism of $Y$. The system $(Y,g)$ is called a {\bf {discretized Anosov  flow}} if it is partially hyperbolic with $\dim E^c = 1,$ dynamically coherent and the center leaves $\ov W^c (y) $ are invariant by $g$. \end{definition}

The basic example of an Anosov flow   is the geodesic flow on a manifold of negative sectional curvature and the basic examples of a discretized Anosov flow are the time-one mapping of an Anosov flow and its small enough perturbations.

Recall that a partially hyperbolic system is called {\bf{accessible}} if for any two points \(y,z\) there exists a path \( x_0 =y, \ldots, x_j, \dots x_k= z \) such that for each \( j, 0\le j < k, \, x_j \) and \( x_{j+1} \) are in the same \( W^\ast \) leaf, with \( \ast = s, u. \)

 Let $(Y,g)$ be a discretized Anosov flow and assume that the system $(Y,g) $ is accessible and that $ g$ has no fixed point.  Non-compact leaves are homeomorphic to $\R$. Since $g$ has no fixed point, we can orientate each 
non compact leaf so that $y<g(y).$ This orientation is preserved by stable and unstable holonomies and is globally defined on non compact leaves. It extends by accessibility to the compact leaves as well. The space $Z$ of the segments $[y, g(y)) $ is a compact space. The mapping $\psi : Z\to Z  $ defined by $\psi ([y,g(y) ) ) = [g(y), g^2(y) )$ is an expansive  homeomorphism with expanding and contracting canonical coordinates. It follows from  \cite{R83} that the system $(Z,\psi ) $ is a Smale space in the sense of D. Ruelle \cite[chapter 7]{R04}. 

\begin{theorem}[\cite{AVW}] Let $(Y,g)$ be a discretized Anosov flow and assume that the system $(Y,g) $ is accessible and that g has no fixed point. Let $m$ be a  $g$-invariant ergodic measure and assume that the projection of $m$ to $Z$ has a $(s,u)$-product structure. Then
\begin{itemize}\item either $g$ is the time-one map of a \(m\)-preserving  flow,
 \item or there exist an integer $k$ and a subset $E \subset Y$ of full $m $-measure that cuts each \( \W^c \)-leaf in exactly $k$ orbits of $g$.
\end{itemize}  
\end{theorem}
\begin{proof}
 Let $m$ be an ergodic  $g$-invariant probability measure. The conditional  measures $\MMM_y $ of $m$ associated to $\ov \W^c$ are defined up to a multiplicative constant. We define, for $m$-a.e. \( y \in Y,\) the measure \( m _y \) on \(\ov W^c(y) \) by normalizing the conditional measure so that \(  m_y ( [y, g(y)) ) = 1.\) 

Consider the central exponent $\g$:
\[ \g   \;:=\;  \lim\limits _{n \to \infty } \frac{1}{n} \log \| D_y g^n v  \|  \; {\textrm { for }} \;  v\not = 0 \in E^c_y .\]

The case  \( \g \not= 0\) has been studied by Ruelle and Wilkinson (\cite{RW01}:  if $\g <0$, say, points in \([y,g(y) )\) are attracted to some  points and  the measures $m _y$ are discrete. By ergodicity,   the measures $m_y $  have to be  an average of a constant number $k$ of Dirac measures. The set $E$ of the second alternative is made of their supports.

We are also in the second alternative if $\g =0 $ and almost all $m_y $ have a discrete component.  By ergodicity, almost all $m_y $ are  an average of a constant number $k$ of Dirac measures.  So we may assume that $\g = 0$ and that, for $m$-a.e. $y, \; m_y $ is continuous. By theorem \ref{TY}, the mappings $y \mapsto m_y$ are $m$-a.e.  equivariant by the local holonomies $h^u$ and $h^s$. Since $m$ has a $(s,u)$-product structure on $Z$,  we see that $m$ has a $(s,c,u)$-product structure (recall proposition \ref{product}).

We can parametrize each universal cover  \( \wt W^c _y \) by \( t_y : \wt W^c_y \to \R, \) \( t_y (z) := m_y ([y,z))\)  (or \( t_y (z) := - m_y ([z,y)) ).\) 
 By uniqueness of the conditional measure, for $m$-a.e. $y,$ if \(z,w \in \wt W^c_y,\) then \(t_z(w) - t_y(w) = t_y (z)\).

 Observe  that, when we parametrize the  (universal cover of the) center leaf of  $y$ by \( t_y \)  for all $y$, the central holonomy  \(h ^c_{y,z} : \wt W^c_y  \to \wt W^c_z \) is the translation  \( R_t \) by \(t = t_y (z) \).  Using accessibility and the $(s,u)$-product structure on $Z$, we conclude that, for $m $-a.e. $y \in Y$, the measure \( ( t_y) _\ast m _y \) on $\R$ is invariant under the translation $R_t$, for \( ( t_y) _\ast m _y \)-a.e. $t.$ 
 The only continuous measures on $\R$ with this property are proportional to  the Lebesgue measure. With our normalization, for $m $-a.e. $y \in Y$,  the measure \( ( t_y) _\ast m _y \) is Lebesgue.  The one-parameter flow \( \psi _t, t\in \R ,\) defined by the translations along the parameter $t_y $ of \( W^c_y \) is $\la $-a.e. well-defined and preserves \(m.\) The transformation $g$ is indeed the time-one of a (measurable, a.e. defined) $m$-preserving one-parameter flow.
 \end{proof}
 
 See \cite{AVW}, \cite{AVW2}, section 9, for further regularity of the flow in the case when $(Y,g)$ is a volume preserving discretized  Anosov flow in dimension 3. Orientability of $\ov \W^c$ is also proven in \cite{AVW2}, section 9,  even in the case of fixed points for $g$.
 
\subsection{Quasi-isometric center leaves}\label{QIcenter}

\begin{definition} Let $(Y,f)$ be a dynamically coherent partially hyperbolic smooth system. We say that $(Y,f)$ has {\bf{quasi-isometric centrer leaves}} if there exists numbers $K_0, C_0,$ such that, if $x,y$ are on the same $ \ov \W ^c $ leaf, we have, for all $n\in Z,$
\begin{equation}\label{Quasi-iso} K_0^{-1} d^c (x,y) - C_0 \; \le \; d^c(f^n x, f^n y) \; \le \; K_0 d^c (x,y) + C_0, \end{equation}
where $d^c$ is the Riemannian distance on the $\ov W^c $ leaves. \end{definition}

Systems with compact center leaves and discretized Anosov flows have quasi-isometric center. leaves.  Let $m$ be an ergodic $f$-invariant probability measure.  We are interested in showing $h^u $ and $h^s $ invariance of the conditional measures $\MMM_y^c$.

We will consider the $m$-a.e. defined and constant exponents 
\[ \g^+_c   := \lim\limits _{n \to \infty } \frac{1}{n} \log \| D_y f^n |_{E^c_y} \| \quad  \g^-_c   := - \lim\limits _{n \to \infty } \frac{1}{n} \log \| (D_y f^n )^{-1}|_{E^c_y} \|.\] If the mapping $f$ has quasi-isometric center leaves, then $\g^\pm _c = 0.$ 
 The following result is obtained in a recent preprint  by S. Crovisier and M. Poletti.
\begin{theorem}[\cite{CP}]\label{autreIP} Let $(Y,f)$ be a dynamically coherent partially hyperbolic smooth system, with quasi-isometric center leaves. Let $m$ be an ergodic $f$-invariant probability measure on $Y$. Then, $m$ has a partial $(s,c,u)$-product structure.
\end{theorem}
\begin{proof} We  show that the conditionals $m^{cu}_y$ on local $\W^{cu}$ leaves have a   $(c,u)$-product  structure for $m$-a.e.$y$. The proof is the same for the $(s,c)$-product structure property of the conditionals $m_y^{cs}$. 

The first observation is a  result in \cite{CP}, that the problem makes sense: local center unstable leaves are nicely foliated by local unstable leaves and local center leaves. More precisely:
\begin{proposition}[\cite{CP} Theorem C] \label{CPproduct} Let $(Y,f)$ be a dynamically coherent partially hyperbolic $C^{1+\alpha}$ system, with quasi-isometric center leaves. Let $m$ be a $f$-invariant probability measure on $Y$. Then, there exists  a full measure set $Z \subset Y$ such that for $y,z \in Z$ with $W^{cu}_z= W^{cu}_y,$ the following holds:\\
For any $w \in \ov W^c_y ,$ the leaves $W^u_w, \ov W^c_z$ intersect at a unique point, denoted $h^u_{y,z} (w) .$ The map $h^u_{y,z} : \ov W^c_y \to \ov W^c_z $ is a homeomorphism. 
\end{proposition}

Let $\pi : \wt Y \to Y$ be the continuous bundle over $Y$ defined by \( \wt Y= \{ (y,w), w \in \ov W^c_y \}. \) We set $p: \wt Y \to Y $ for the projection $p((y,w)) = w.$  The fiber above $y$ is endowed with the metric $d_y$ coming  from the Riemannian metric on $ \ov W^c_y.$  Let $ F$ be the bundle diffeomorphism of $\wt Y$  defined by  $ F (y,w) := (fy,fw).$ We have $\pi \circ F = f \circ \pi, p\circ F = f \circ p.$
For $t>0$, write $\wt Y_t := \{ (y, w), w \in W^c_y; d(y,w) \le t \} .$

Assume  that $f$ has quasi-isometric center leaves with constants $K_0,C_0.$ Then, for $t > K_0 C_0,$
\begin{equation}\label{quasi} \wt Y_{K_0^{-1}  t -C_0 } \; \subset \;\cap _{n \in \Z}  F^n \wt Y_t  \; \subset \; \cup _{n \in \Z}  F^n \wt Y_t  \; \subset  \; \wt Y_{K_0 t +C_0} .\end{equation}
Set $ \EE _t := \cap _{n \in \Z} F^n \wt Y_t  .$  The set $\EE _t$ is closed and  $ F$-invariant. If the center leaves are compact, for $t$ large enough,  the set $\EE_t$ is the whole $\wt Y.$  Otherwise, by (\ref{quasi}), for $t > K_0 C_0,$ the set $\EE _t$ is compact and the intersection $\EE_t(y)$ with the  fiber at each $y$ has non-empty interior. Moreover, for $K_0 C_0 <t <t', \; \EE_t \subset \EE_{t'}.$

Recall that the conditional measures $\MMM_y^c$ on $\ov \W^c$ leaves are defined up to multiplication by a constant on the leaf. Such equivalence classes are invariant  by $f$. The union of the supports of the measures $\MMM_y^c$  is an invariant set  $Z'$ of full measure. For $y \in \pi (Z')$, we endow 
the fiber above $y$ with the  measure $M_y $ proportional to $ \MMM_y^c,$ normalized by $M_y ( \EE _{K_0 C_0 +t_0} ) = 1,$ for some $t_0 >0.$
The measure $M:= \int M_y \, dm (y) $ is $F$-invariant  and $M ( \EE _{K_0 C_0 +t_0} )  = 1.$ Observe that we can choose $t_0$ in such a way that the boundary of $ \EE _{K_0 C_0 +t_0} $ is negligible. By definition $\pi _\ast M = m.$ We also have
\begin{lemma}   With the above notations, $p_\ast M = m$. \end{lemma}
\begin{proof} For each $z \in Y,$ $$p^{-1} (z) = \{ (y,z), y \in \ov W^c_z, d(f^n y, f^n z) \le K_0 C_0 +  t_0, \, \forall n \in \Z\}.$$ Locally, the conditional measures $M_{z'}$ of  the measure $M$ on suitable $\pi ^{-1} (\ov W^c_{loc} (z')) $ containing $z$  are  proportional to $\int _{\ov W^c_{loc} (z')} dM_{y} (w)\,  d\MMM _{z'}(y).$ Therefore
\begin{eqnarray*} & & M_{z'}(p^{-1} B_{\ov W^c} (z, \e)) \,=\, C(z') \int _{\ov W^c_{loc} (z')} M_{y} (p^{-1} B_{\ov W^c} (z, \e) )\,  d\MMM _{z'}^c(y)\\ &=&  C(z')\int _{\{y: d(f^n y, f^n z) \le K_0 C_0 + t_0 \, \forall n \in \Z\} } \frac {\MMM_{y}^c (B_{\ov W^c} (z, \e) )}{\MMM _y^c (\EE _{K_0 C_0 + t_0} (y))} \,  d\MMM _{z'}^c(y).\end{eqnarray*}
Letting $\e \to 0,$ we obtain that the measure $p_\ast M_{z'} $ has the same negligible sets as the measure $\MMM _{z'}^c.$ Integrating in $z'$ shows that the measure $p_\ast M $ has the same negligible sets as the measure $m$. Since $p_\ast M$ is a $f$-invariant probability measure and $m$ is ergodic, we indeed have $p_\ast M = m.$
\end{proof}

Set $\EE := \EE _{K_0 C_0 +t_0} .$ We will consider the system $(\EE, F, M)$. 
\begin{lemma} $(\EE, M)$ is a  compact manifold with negligible boundary. The mapping $F: \EE \to \EE$ preserves the boundary, is $C^{1+\alpha}$ and
 partially hyperbolic with decomposition $\wt E^s \oplus \wt E^c \oplus  \wt E^u ,$ with
\[ D\pi  \wt E^s = E^s, \; \wt E^c := (E^c \otimes \{0\}) \oplus (\{0\} \otimes T\ov W^c), \; D\pi \wt E^u = E^u . \] \end{lemma}
\begin{proof} Only the regularity of $F$ is not immediate from the construction and proposition \ref{CPproduct}. But the decomposition is H\"older continuous and the derivative $DF$ is diagonal in this decomposition, with H\"older continuous components. \end{proof}

The proof of theorem \ref{autreIP} follows the scheme of the proof of theorem \ref{TY}: let $\zeta ^u $ be a measurable partition of $Y$ subordinated to the unstable foliation of $(Y,f)$ and increasing under $f$.   
By proposition \ref{basicinequality3}, the partition $\zeta ^u$ is such that $H (f^{-1}\zeta ^u| \zeta ^u) $ is the entropy of $f$.  By proposition \ref{perfect}, there is a finite entropy measurable partition $\CC$ such that $\zeta ^u = \vee _{n\ge 0} f^n \CC .$
We set, for $ (y,w)  \in \EE,$
 \[   \xi ^{cu} (y,w) := \pi ^{-1} (\zeta ^u (\pi (y)) \; {\textrm {and} } \; \xi ^{u} (y,w) := \xi^{cu} (y,w) \cap p^{-1} (W^u_{loc} (p (y)))  .\]
 Observe that $\xi  ^{cu} = \vee _{n\ge 0} F^n (\pi ^{-1}  \CC ) $ is an increasing measurable partition. 

 \begin{lemma} The partition $\xi ^u $ is an increasing  measurable partition  that refines $\xi ^{cu}$ and which is subordinated to the unstable foliation $\wt W ^u$. \end{lemma}
 \begin{proof} We first have to verify that  $\xi ^u (y,w) $ is well defined $M$-a.e.. More precisely, for $(y,w) \in \EE$  and $z \in \zeta ^u (y),$ we have to find a unique $w'\in \EE _{K_0 C_0 +t_0} (z) \subset \ov W^c_{z}$ such that $w' = p(y',w') \in W^u_{loc} (w) $. By proposition \ref{CPproduct}, we can take $w' := h^u_{y,z} (w) $ if $y$ and $z$ are in $ Z$ and also if this $h^u_{y,z} (w) $ belongs to $\EE _{K_0 C_0 +t_0}(z) .$ 
 
 The set of $(y,w) \in \EE$ such that $y \in Z$ has full $M$ measure and for $m$-a.e. $y \in Y,$ the set of $z \in \zeta ^u (y) $ that belong to $Z$ has full conditional $m^u_y$ measure.
 So $w'=  h^u_{y,z} (w) $ is defined for $M$-a.e. $(y,w) $ and $m^u_y$-a.e.$z \in \zeta ^u (y)$ and belongs to $W^u_{loc} (w).$ It follows that, for $M$-a.e. $(y,w) \in \EE,$  $\xi ^u (y,w)$
 is an open subset of $p^{-1} (W^u_{loc} (w)).$  This open subset is non-empty since for $M$-a.e. $(y,w) \in \EE,$ $w $ lies in the interior of $\EE_{K_0 C_0 +t_0} (y).$ Furthermore, since $pF=fp,$  $p^{-1} (W^u_{loc} (p(y,w)))$ is an open neighborhood of $(y,w) $ in $\wt W^u_{loc} (y,w).$
 
 It is clear that the partition $\xi ^u$ is increasing and refines $\xi ^{cu}.$ It remains to see that the partition $\xi ^u $ is a measurable partition. The partition $\xi ^{cu}$ is a measurable partition and, using local charts, on can define  measurable partitions $\zeta $ of $Y$ such that  locally $W^u _{loc} (w) = \zeta (w).$ Since the projection $p$ has compact fibers, the partitions $p^{-1} \zeta $ are measurable partitions on a piece of $\EE$ and $\xi ^u = \xi ^{cu} \vee p^{-1} \zeta $ is a measurable partition. \end{proof}
\begin{lemma}\label{parry} Let $\xi ^u $ be defined as above. There is a finite entropy partition  $\CC '$ such that $F^{-1} \xi ^u = \CC' \vee \xi ^u.$ \end{lemma}
\begin{proof} Observe that, since the partition $\xi ^u $ is subordinated to $\wt W^u,$ the entropy $H( F^{-1} \xi ^u | \xi ^u ) $ is the entropy of the  system $(\EE, F, M)$.
Indeed, this equality follows from the $C^{1+\alpha}$ version of proposition \ref{basicinequality3} (\cite{B22}). Then,  the lemma follows from  proposition \ref{perfect}. \end{proof}
\begin{proposition}\label{invariance5} We have $H(\CC '| \xi ^u ) \le H(\CC '| \xi ^{cu} ) $ with equality if and only if, for $M$-a.e. $(y,w) \in \EE$, the conditionals $M_{y,w}^{cu} $ are $(c,u)$-products. \end{proposition}
\begin{proof}The proof is the same as the proof of  proposition \ref{invariance4}. The only modification is that the partition $\CC'$ has finite entropy instead of being finite. This implies that, for $M$-a.e. $(y,w)$, the entropy of the partition $\CC'$ for the measure $M_{y,w}^{cu} $ is finite. The inequality still follows by summing over the atoms of $\CC'$. In the case of equality, 
we also have equality for individual atoms at $M$-a.e.$(y,w)$, since the series converges.
We  have $H(F^{-n} \xi ^u | \xi ^u )  =n H(F^{-1} \xi ^u | \xi ^u ) \le n  H(F^{-1} \xi ^{cu} | \xi ^{cu} ) = H(F^{-n} \xi ^{cu} | \xi ^{cu} ) $ for all $n >0$. By lemma \ref{parry}, $F^{-n} \xi ^u =\vee _{j=0}^{n -1} F^{-j} \CC' \vee \xi ^u.$ The partition  $\vee _{j=0}^{n -1} F^{-j} \CC'$ has finite entropy and in case of equality, we still have that, for $M$-a.e. $(y,w)$, the measures $M_{y,w}^u $ and $M_{y,w} ^{cu} $ coincide on $\vee _{j=0}^{n -1} F^{-j} \CC '.$ Since this holds for all $n\in \N,$ the $(c,u)$-product property follows if  we have equality  $H(F^{-1} \xi ^u | \xi ^u ) = H(F^{-1} \xi ^{cu} | \xi ^{cu} ) .$   
\end{proof}
Observe that  the  equality in proposition \ref{invariance5} holds because $H(F^{-1} \xi ^u | \xi ^u ) $ is the entropy  of the  system $(\EE, F, M)$ and $ H(F^{-1} \xi ^{cu} | \xi ^{cu} ) $ is at most the entropy since the partition $\xi ^{cu} $ is increasing (see (\ref{KSentropy})). So, by proposition \ref{invariance5}, the conditionals $M_{y,w}^{cu} $ are $(c,u)$-products.  In particular, the local unstable manifolds $\wt W^u_{loc}$ preserve the conditional of $M$ along the fibers $\pi ^{-1} \{y\}.$ In other words, the local holonomies $h^u_{y,z} $ preserve the measure $\MMM _y $ up to a constant, on a set containing $\{w \in W^c_y, d_{W^c} (y,w) < t_0  \} .$
Since this property holds for arbitrary large  $t_0$, this shows that  that the conditionals $m^{uc}_y$ on local $\W^{uc}$ leaves have a   $(c,u)$-product structure  for $m$-a.e. $y$.
\end{proof}

 \section{Miscellaneous}

 In most developments, the Invariance Principle is a tool that is used in conjunction with other ideas. The examples of sections \ref{examples2} and \ref{examples3} are somehow the most direct applications of the IP. In this section, we mention possible further extensions (see also \cite{MP25}).
 
 An important feature is remark \ref{perturbations}. It often implies that the set of objects controlled by an IP is stable by perturbations. Details are delicate and depend on the problem, but this stability is one extra feature of  the applications of  the IP.
  
 \subsection{Stochastic flows} A common source of random diffeomorphisms is the flow of random diffeomorphisms associated to a Stochastic (Partial) Differential Equation. In general, the reference space need not be locally compact, but is still a separable metric space, for instance, some Sobolev space of functions on a domain. See  for example \cite{BBPS} for the formulation and the applications of the Invariance Principle in that setting.
 
 The flow of random diffeomorphisms associated to a Stochastic Ordinary Differential Equation connects the theory of random action of  (pseudo-)groups of diffeomorphisms with the  properties of  diffusions on the leaves of a foliation (see \cite{DK07}). Indeed, the discrete  Markov process defined by the successive hits of  a transversal by the trajectory of the leafwise diffusion can be obtained as the Markov process  associated to a random diffeomorphism of the transversal. In many cases, it is possible to define a Lyapunov exponent for the random product and vanishing of this exponent is associated to the foliation admitting an invariant transversal measure.  
\subsection{$(s,u)$-product structure} If the system $(Y,f,m)$ is partially hyperbolic, it might be important to assume that  the invariant measure $m$ has a $(s,u)$-product structure. Then, the IP might provide a partial $(s,c,u)$ product structure (see e.g. section \ref{PHcompact}), and proposition \ref{product} would yield  a $(s,c,u)$-product structure. This is the case for measures of maximal entropy and equilibrium measures for a H\"older continuous potential (see e.g. \cite{RHRHTU}, \cite{UVY}, \cite{CP}), or a Lebesgue measure when it is invariant, or physical measures (see e.g.  \cite{VY13}, \cite{CP}), or else u-Gibbs measures (\cite{UVYY}).

\subsection{Distinct exponents}\label{distinct} Consider an independent product of matrices in $\SL(d,\R),$ $ d>2,$ with distribution $\mu$. By the subadditive ergodic theorem, there are exponents
\[ \la _1 \ge \la _2 \ge \ldots \ge \la_d , \; {\textrm { with }} \sum _j \la _j = 0 .\]
The analog of theorem \ref{F} gives conditions to ensure that $\la _1 > \la _d $ and all the discussion above concerns possible extensions and variants. It might be of interest to decide whether  exponents are all pairwise  distinct. In the independent case, the known criteria involve the support of the measure $\mu.$ Y. Guivarc'h and A. Raugi (\cite{GR89}) showed that this is the case as soon as there one matrix in the semi-group generated by the support of $\mu $ with moduli of eigenvalues all distinct. I. Gol'dshe\u{\i}d and G. Margulis (\cite{GM89}) showed that it suffices that the semi-group generated by the support of $\mu $ is Zariski dense in $\SL(d,\R).$ There isn't a generalization of the criteria    of \cite{GR89} and \cite{GM89} as neat as  theorem \ref{IP}, but it is still possible to obtain nice results, in particular:

The proof of Zorich-Kontsevich conjecture  by A. Avila and M. Viana (\cite{AV07}): on each transitive stratum of the Teichm\"uller flow, for the Masur-Veech invariant probability measure, the Zorich-Kontsevich cocycle has  exponents pairwise distinct. See e.g \cite{F12}  for background and consequences, in particular of $\la_1 > \la _2,$ which was proven earlier by other methods (see \cite{F02}).

 The analog of section \ref{positivity}  for  generic $C^\infty $ linear cocycles or for the derivative cocycle of a generic $C^\infty  $  diffeomorphism:  either all exponents are the same, or exponents are pairwise distinct and the Oseledets splitting is dominated. See e.g. \cite{bochiviana}  \cite{AB12}, \cite{RH13}, \cite{BPLV}, \cite{Mit22} and the references therein.

A more general result for cocycles is proven in a recent preprint of U. Bader and A. Furman. 
\begin{theorem}[\cite{BF25}]\label{BF} Let $(X,\A,m;f)$ be a measurable dynamical system, $\G$ a locally compact second countable group, $w :X \to \G$ a mapping such that the $\s$-algebras generated by $w\circ f^n, n\ge 0 $ and $w \circ f^n, n\le -1,$  are quasi-independent (see section \ref{quasi-independence}).
Let $G$ be a connected semi-simple real Lie group with finite center. 
Then, for any representation $\rho : \G \to G$ with a Zariski dense image and satisfying the suitable integrability condition, the cocycle $\rho \circ w : X \to G$ has a simple Oseledets spectrum. \end{theorem}

The theorem holds with an asymptotic condition of quasi-independence (which amounts to say, in a different language, that the process $ \{\rho \circ w \circ f^n\}_{n \in \Z} $ is Weak Bernoulli). One can observe that, even in the case of $\SL (2,\R)$, the  result is slightly more comprehensive than theorem \ref{holdercocycle}. In higher rank, theorem \ref{IP} can be applied, in the setting of theorem \ref{BF},  to the action of $\rho \circ w (x)$ on the different flag spaces simultaneously, but, a priori, this might  need a still heavier machinery than \cite{BF25}.

\appendix
\section{Background results in ergodic theory}                

In this Appendix, we recall notions and properties from measure theory and ergodic theory that are used in the text. 
\subsection{Lebesgue space, measurable partition, disintegration}\label{Rokhlin}
A complete probability space $(X, \A, m)$ is called a Lebesgue space if there exist a measurable set $E$ with full measure and a countable family of measurable sets $\{C_j \}_{j \in \N} $ such that any point  $x \in E$ is the intersection of $E$ and the sets $C_j$ that contain $x$.  A Polish metric space endowed with a 
Borel probability measure and the completed Borel $\s$-algebra is a Lebesgue space.

Let $(X, \A, m)$ be  a Lebesgue space. A partition $\xi $ of $X$ is called a {\it{measurable partition}} if there exist a measurable set $F$ with full measure and a countable family of measurable sets $\{D_j \}_{j \in \N} $ such that for  any element $P \in \xi $, $P\cap F $  is the intersection of $F$ and the sets $D_j$ that contain $P$. Let $\xi $ be a measurable partition and denote $(\wh X, \wh \A ,\wh m) $ the quotient space. By definition, $(\wh X, \wh \A ,\wh m) $ is a Lebesgue space. The key property of measurable partitions is that they admit {\it{disintegrations}}.
\begin{theorem}(\cite{R49}) Let $\xi $ be a measurable partition of the Lebesgue space $(X, \A, m)$,   $(\wh X, \wh \A ,\wh m) $ the quotient space. There exists a family 
$P \mapsto m_\xi ^P $ of  probability measures on $X$ such that:
\begin{enumerate}\item for $\wh m$-a.e. $P \in \xi,$ $ m_\xi ^P (P) = 1 $,
\item for any $\A$ measurable non-negative function $\psi $ on $X$, the function $P \mapsto  \int_P \psi (x) \, dm_\xi ^P (x) $ is $\wh \A $ measurable and 
\item $\int _X \psi \, dm  = \int _{\wh X} \left( \int_P \psi (x) \, dm_\xi ^P (x)\right) \, d \wh m (P) .$
\end{enumerate}
The family $m_\xi ^P$ is $\wh m$-essentially unique and is called the disintegration of $m$ with respect to  the partition $\xi .$
\end{theorem}
In this paper, we met different kinds of measurable partitions. If we have  a measurable fibration $\pi : X \to \wh X,$ then the  partition in fibers $\pi ^{-1} \wh x, \wh x \in \wh X$ form a measurable partition if the quotient space is a Lebesgue space. 
In the case of a continuous  foliation $\W$ of a Polish space   with  compact center leaves,  the partition defined by the leaves  is measurable and the associated  conditional measures  
 $\MMM _{\wh x} $ are defined $\wh m$-a.e. (\cite[Proposition 3.7]{AVW2}). In the general case of non-compact  leaves, the partition defined by the leaves  may not be  measurable and we can only define local conditional measures on flow boxes for the foliation. It turns out that these conditional measures are proportional to one another on the overlap of different flow-boxes.  It follows that there is a family of Radon measures on almost all  universal covers of the leaves, defined up to multiplication by a constant, that are proportional to the local disintegrations (see \cite{AVW}). We will still denote $\MMM_{\wh x} $ this   family of Radon measures on the leaves,  up to multiplication by a constant, with the understanding that  they are proportional to the local disintegrations  on simply connected pieces and we call them conditional measures on the leaves of $\W .$

Observe that,  if $(X, \A, m)$ is   a Lebesgue space and $\B \subset \A$ is a sub-$\s$-algebra, there exists a measurable partition $\xi$ such that the disintegration $P \mapsto  \int \psi \, dm_\xi ^P $ gives the conditional expectation of the function $\psi $ with respect to $\B$, but, in general, the elements of $\xi $ are not the atoms of $\B$ and, in particular for foliations with non-compact leaves,  the elements of $\xi $ are non-trivial union of leaves  and the  disintegration is  different from the conditional measures of the previous paragraph.

\subsection{Entropy}
Let $(X,\A, m)$ be a Lebesgue space, $f$ an invertible bimeasurable transformation preserving $m$, $\xi$ a measurable partition. The partition $\xi $ is said to be increasing if the partition $f^{-1} \xi $ is finer than the partition $\xi .$ In that case, the entropy $ H(f ^{-1} \xi |\xi )$ is defined by the following (cf. formula (\ref{conditionalentropy}))
\begin{proposition}\label{entropyrohlin}[\cite{R67}] Assume there exists a countable partition $ \CC = \{C_\ell\}_{\ell  \in \N} $ with finite entropy such that $f^{-1} \xi = \CC \vee \xi.$ Then the quantity 
\[  H(f ^{-1} \xi |\xi ) := \int -\sum _{ \ell }  m_\xi ^x (C_{\ell})\log  m_\xi ^x (C_{\ell}) \, dm (x) ,\]
where $m_\xi ^x$ is a family of disintegrations of $m$ with respect to $\xi$, does not depend on the partition $\CC$ such that  $f^{-1} \xi = \CC \vee \xi.$ \end {proposition}
We set $H(f ^{-1} \xi |\xi ) : = +\infty $ if there is no finite entropy partition such that  $f^{-1} \xi = \CC \vee \xi.$

The entropy $h_m (f) $ is given by (see \cite {R67})
\begin{equation}\label{KSentropy}  h_m(f) \;:=  \;\sup _{\xi, f ^{-1} \xi 
\prec \xi } H(f ^{-1} \xi |\xi ). \end{equation}

Assume $h_m(f) < + \infty .$ An increasing measurable partition $\xi $ is called {\it{perfect}} if it is generating and $ H(f ^{-1} \xi |\xi)  = h_m(f).$ If $\CC = \{ C_j, j\in \N \} $ is a countable measurable partition with 
finite entropy that generates $\A$, then the partition $\xi := \vee _{j\ge 0} f^j \CC $ is perfect. Conversely
\begin{proposition}\label{perfect}(\cite[statement 10.11]{R67}, \cite[Theorem 8.9]{P69}) Let  $(X,\A, m,f)$ be a finite entropy dynamical system and $\xi $ a perfect increasing partition. Then, there exist a partition  $\CC = \{ C_j, j\in \N \}$ with finite entropy such that  $\xi = \vee _{j\ge 0} f^j \CC .$ \end{proposition}

\subsection{Averaging in Besicovich spaces}\label{Besicovich}
\begin{definition} A metric space $X$ is called a Besicovich space if there exists a constant $C$ with the following property: given any function $x \mapsto r(x) >0 $ on $X$ such that 
$X  \subset  \cup _x B(x,r(x)),$ there is a subset $Y \subset X$ such that  $X  \subset  \cup _{y \in Y} B(y,r(y)),$ and no point $x \in X$ is covered by more than $C$ balls $B(y,r(y)), y \in Y.$
\end{definition}
Subsets of Euclidean spaces are Besicovich, as well as metric spaces that are bilipschitz equivalents to subsets of Euclidean spaces. A Borel probability measure $m$ on a Besicovich space has the Lebesgue density property and the averages on balls have the weak (1,1) property. For us, this means that we have the following
\begin{proposition}\label{densities} Let $X$ be a Besicovich space, $m,m'$ Borel probability measures on $X$, $g $ the $m$-a.e. defined Radon-Nikodym derivative $g= \frac{dm'}{dm}.$
For   $\de >0,$ set 
\[ g_\de ( x) \;: =\; \frac {\int _{B(x, \de)} g( x) \, dm}{m (B(x,\de))} =  \frac {m' (B(x, \de))}{m (B(x,\de))},\]
Then, for $m$-a.e. $x$, $\lim\limits _{\de \to 0} g_\de (x) = g(x).$ 
Moreover, if we set, for $N$ large, \[ g^N ( x) := \sup \{g(x), e^{-N}\}, \quad
g_\de^N ( x) \;:=\;  \frac {\int _{B (x, \de)} g^N( y) \, dm (y)}{m(B(x,\de))}  \;\; {\textrm {and}} \;\; 
 g_\de^\ast := \sup _{\de ' \le \de } g_{\de '}^N,\]
 the  functions $-\log g^N,  -\log g_\de^\ast $ are  $m$-integrable, we have \( -\log g \ge - \log g^N \ge -\log g _\de ^\ast \) and  
\[  \lim\limits_{\de \to 0} \int \log \frac {g_\de^\ast}{g^N} \, dm =0 .\]
\end{proposition}
See e.g.  \cite{lessa}, lemmas 8, 9 and 10 for a proof of these classical facts.

In dimension 1, we can take all intervals as Besicovich sets,  so that proposition \ref{densities} holds with 
\begin{equation}\label{densitiesdim1}   g_\de^\ast (x) \; := \;\sup _{I \ni x, |I|\le \de }  \frac {\int _{I} g^N( y) \, dm (y)}{m(I))} . \end{equation}

 In particular,
 \begin{corollary}\label{density}Let $m$ be a probability measure on $\SSSS^1$. Define, for $x \in \SSSS^1,$
 \[ Q(x) := \sup_{I \ni x} \frac {|I|}{m(I)} .\] 
 Then, $\log ^+ Q\in L^1 (m).$ \end{corollary}
%

\subsection{Quasi-independence}\label{quasi-independence}

 Two random variables $Y$ and $Z$ are said to be quasi-independent if the distribution of $(Y,Z)$ has the same negligible sets as the product of the distributions of $Y$ and $Z$\footnote{Quasi-independence of the past and the future of a stationary sequence is a qualitative property of independence which appears in ergodic theory; compare here with the argument in section  \ref{comments} and in the proof of proposition \ref{fiberinvariance}. }  . They are said to be quasi-independent  conditionally with respect to a third random variable $X$ if, for a.e. value  $x$, $Y$ and $Z$ are quasi-independent for the conditional distributions $P_x(dy) $ and $P_x(dz)$. 
\begin{lemma}\label{quasiindependence} Let $X,Y,Z$ be random variables defined on a Lebesgue space $(\Om ,P)$. Assume that $X$ and $Y$ are quasi-independent  conditionally with respect to $Z$, that $X$ and $Z$ are quasi-independent  conditionally with respect to $Y$ and that $Y$ and $Z$ are quasi-independent. Then, $X, Y$ and $Z$ are quasi-independent in the sense that the distribution of $(X,Y,Z)$ has the same negligible sets as the product of the distributions of $X, Y$ and $Z$. \end{lemma}
\begin{proof} 
Let $\{\A_n\}_{n \ge 0}, \{\B_m\}_{m\ge 0}, \{\CC _p\}_{p\ge 0} $ be finer and finer sequences of finite partitions generating $X,Y$ and $Z$. Our hypothesis is that the following limits exist and are non-zero $P$-a.e. as $n,m$ and $p$ go to infinity:
\begin{equation*}   \lim\limits _{n,m,p \to \infty} \frac {P(A_n\cap B_m) P(C_p \cap B_m)}{P(A_n\cap B_m \cap C_p)P(B_m)} ,\;
  \lim\limits _{n,m,p \to \infty} \frac {P(A_n\cap C_p) P(C_p \cap B_m)}{P(A_n\cap B_m \cap C_p)P(C_p)}
\end{equation*}
and \( \displaystyle \lim\limits _{m,p \to \infty} \frac {P(B_m)P(C_p)} { P(B_m \cap C_p)}.\)
After product and simplifications, it reads as the following limit exist and is non-zero $P$-a.e. as $n,m$ and $p$ go to infinity:
\[\lim\limits _{n,m,p \to \infty} \frac {P(A_n\cup B_m) P(A_n \cap C_p)}{P(A_n\cup B_m \cup C_p)P(A_n)}\;  \frac  {P(B_m) P( C_p)P(A_n)}{P(A_n\cup B_m \cup C_p)}.\]
Both ratios converge toward densities, which are positive  $P$-a.e. since their product does not vanish $P$-a.e..\footnote{Observe that  the lemma with independence instead of quasi-independence is also true. The computation is the same, and now the product of the densities is 1. But by Jensen inequality, the integral of the logarithms of both  densities is non-positive. So both densities are 1 and the three variables are independent.} Thus, 
\( \displaystyle \lim\limits _{n,m,p \to \infty}  \frac  {P(B_m )P(C_p)P(A_n)}{P(A_n\cup B_m \cup C_p)} \)  is non-zero $P$-a.e..
\end{proof}

\subsection{Pliss Lemma}
We recall the following lemma, due to V. Pliss (\cite{P72})

\begin{lemma}\label{Pliss} Given $A \ge c_2 > c_1 > 0,$ let $\beta  := (c_2  -c_1)/(A - c_1).$ Then,
given any real numbers $a_1,\ldots  , a_N$  such that
\[ \sum _{j=1}^N a_j \ge  c_2 N {\textrm { and }} \; a_j \le  A  \; {\textrm { for every }}\;  1 \le j \le N, \]
there are $\ell \ge \beta N $ and $1 < n_1 < \ldots < n_\ell \le  N $ so that
\[ \sum _{j= n+1}^{n_i}
a_j \ge  c_1 (n_i - n) \; {\textrm {for every }} \; 0 \le n < n_i \; {\textrm{ and }} \; i = 1,\ldots,\ell. \]
\end{lemma}

See e.g. \cite{ABV00}, lemma 3.1,  for a proof. We will use a dynamical version of Pliss Lemma. Let $(X,\A, m;f)$ be an ergodic transformation of a Lebesgue space, $\vf \in L^\infty (X,m),$ and choose $\e >0.$ An integer $n$ is called $\e$-good for $x$ if for all $j, 0\le j < n,$ we have \[ \sum _{k = j+1}^n \vf (f^k x) \; \ge \; (n-j) \left( \int \vf \, dm -\e \right).\]
\begin{corollary}\label{Pliss2} Let $(X,\A, m;f)$ be an ergodic transformation of a Lebesgue space, $\vf \in L^\infty (X,m),$ and choose $\e >0.$ There exist $\beta >0$ such that, for $m$-a.e. $x \in X,$ the asymptotic upper density of the integers that are $\e$-good for $x$ is at least $\beta.$ \end{corollary}
\begin{proof} By replacing $\vf$ by $\vf + Constant,$ we may assume that $\int \vf  \, dm -2\e >0.$ Then the ergodic theorem  shows that for $m$-a.e. $x$ and for $N$ large enough, the sequence $\vf (f^j x)$ satisfies the hypotheses of lemma \ref{Pliss} with $A = \|\vf \|_\infty, c_2 = \int \vf \, dm - \e /2,  c_1 = \int \vf \, dm - \e. $ So, for $N$ large enough, there are at least  $\beta N$ $\e$-good times for $x$, where $\beta = \frac{\e/2}{ \|\vf \|_\infty - \int \vf \, dm + \e /2} >0.$ \end{proof}

\subsection{Oseledets theorem}
In this subsection, we state the multiplicative ergodic theorem. There is the classical Oseledets theorem \ref{oseledets1} and its two sided version (theorem \ref{oseledets2}).

We consider a probability space \((X, m)\) and an ergodic  measure preserving  transformation \(T\) on \(X\).
\begin{theorem}[\cite{oseledets}]\label{oseledets1} For \( g \in \GL (d, \R) , \) set \( |g| = \log \| g\| ,\) where \(\| g\| \) is the  norm of \(g\) as an operator in the Euclidean space. Let \(g : X \to \GL(d, \R) \) be measurable  and such that  \( \int_{X } |g(x)^{\pm 1}| \, dm (x)  <  + \infty .\)
Set
\[ X_n(x ) := g (T^{n-1}x) \ldots g(x) .\]
There exists numbers \( \la _1 > \ldots > \la _\ell \), with multiplicities \( d_1, \ldots, d_\ell, \sum _{i=1 }^\ell d_i = d,\) and, for \( m \)-a.e. \( x \in X,\) a partial flag
\[ \F'(x) = \{0\}= U_{\ell +1} \subset U_{\ell }(x) \subset \ldots \subset U_{2}(x) \subset U_1 = \R^d ,\]
with \( \dim U_i (x) = \sum _{j\ge i} d_j ,\) such that  \( \sum _{i=1}^\ell d_i  \la _i = \int \log |\det g (x)| \, d m (x)\) and
\begin{equation}\label{growthrate} v \not = 0 \in U_i(x) \; \iff \; \lim\limits_{n\to +\infty} \frac{1}{n} \log \| X_n(x) v\| \le \la _i.\end{equation} \end{theorem}
The flag \(\F'(x)\)  in theorem \ref{oseledets1} is called the stable flag.

If the transformation \(T\) is invertible \(\ov m \)-a.e., we can apply theorem \ref{oseledets1} to the cocycle  \( (X,  m, T^{-1}; g^{-1}\circ T^{-1} ).\) We obtain a flag \(\F(x) \) that is called the unstable flag. The stable flag  and  the unstable flags at \( x\) are in general position and give rise to the Oseledets splitting:
\begin{theorem}[\cite{oseledets}]\label{oseledets2} Assume moreover that the transformation \(T\) is invertible \(m \)-a.e.. Then, there exists at \( m\)-a.e. \( x \in X\) a decomposition \[ \R^d \;=\; F_1(x) \oplus \ldots \oplus F_\ell (x) ,\]
with \( \dim F_i (x) = d_i \, m\)-a.e. such that
\begin{equation}\label{splitting} v \not = 0 \in F_i(x) \; \iff \; \lim\limits_{n\to \pm\infty} \frac{1}{n} \log \| X_n(x) v\| = \la _i,\end{equation}
where, for \(n<0, X_n(x ) := (g(T^{-n}x))^{-1} \ldots (g(T^{-1} x))^{-1}. \) Moreover, for all \( i, 1< i <\ell, \)
\begin{equation}\label{angle} \lim\limits_{n\to \pm\infty} \frac{1}{n} \log | \angle (\oplus _{j\le i} F_j (T^n x) , \oplus _{j> i} F_j (T^n x) )| =0. \end{equation}
\end{theorem}

\

\subsection{Pesin theory}\label{pesin}
In the setting of section \ref{TheoremIP}, for $M$ an ergodic $F$-invariant probability measure,  we can apply theorems \ref{oseledets1} and \ref{oseledets2} to the mapping $g: \EE \to \GL(d,\R) $ obtained  from $D_\xi  F_x$ by some measurable trivialization of the tangent bundles $T\EE_x.$ The flags and splittings that are obtained can be seen as measurable splittings and flags  of the  tangent bundles $T\EE _x.$ Following \cite{P78} (more precisely, following \cite{LQ95} for Pesin theory depending on  a parameter), for any $\e >0$ sufficiently small,  there is a $M$-a.e. positive function $C(x,\xi)$, satisfying $C(F^{\pm 1} (x,\xi)) \le e^{\e} C(x, \xi )$ and, for $n \ge 0$, local charts $\Phi _{n,x,\xi } $ such that the Osseledets splitting is dominated and uniform for  the action of $F_{f^nx}$  in the charts. Moreover,  the regularity of $DF_{f^nx} $ in the charts is controlled. Pesin theory then gives the construction of local unstable and stable manifolds which are as regular as in the uniform case on compact subsets $K_\ell$, with $C(x, \xi) \le \ell $ on $K_\ell$ and $M(\cup _\ell K_\ell ) = 1.$

\begin{proposition}[\cite{LS82}] \label{LS} For any unstable lamination $\W$ coming from Pesin theory, there exist generating, increasing measurable partitions $\zeta $ such that for $M$-a.e. $(x,\xi)$, $\zeta (x,\xi) $ is a neighborhood of $(x,\xi)$ in $W(x,\xi)$. \end{proposition}
A version of the basic inequality is that for a $C^{1+\alpha}$-diffeomorphism, the unstable foliation $\W^u$ carries the whole entropy:
\begin{proposition}[\cite{LY85}, Proposition 5.1 and Corollary 7.2.2 for $\vf $ of class $C^2$,  A. Brown (\cite{B22}) for the $C^{1+\alpha }$ case]\label{basicinequality3}  Let $\vf$ be a $C^{1+\alpha}$ diffeomorphism  of a compact Riemannian manifold $Y$, $\nu$ an invariant ergodic probability measure, $\W^u $   the $\nu$-a.e. defined  Pesin unstable foliation. Let $\zeta ^u $ be an increasing partition made, up to 0 measure, of open  bounded subsets of unstable leaves. Then 
\[ H(\vf ^{-1} \zeta ^u | \zeta ^u) \;=  \; \sup _{\xi, \vf ^{-1} \xi 
\prec \xi } H(\vf ^{-1} \xi |\xi ).\] \end{proposition}

See  S. Gan, Y. Tong and J. Yang  (\cite{GTY}) for this form of the  basic inequality if $\vf$ is only $C^1$  in the presence of dominated splitting and a dimension 1 center direction.

\end{document}